\documentclass[11pt]{article}

\usepackage{array}
\usepackage{amsfonts}
\usepackage{amscd}
\usepackage{amssymb}
\usepackage{amsthm}
\usepackage{amsmath}
\usepackage{stmaryrd}
\usepackage{graphicx}
\usepackage{color}
\usepackage{verbatim}
\usepackage{mathrsfs}
\usepackage{pstricks,tikz}
\usetikzlibrary{calc}
\usepackage{caption}
\usepackage[neveradjust]{paralist}
\usepackage[a4paper, total={15cm, 25cm},centering]{geometry}
\usepackage{float}
\usepackage[absolute]{textpos}
\usepackage{subcaption}
\usepackage{ytableau}
\usepackage{hyperref}

\usetikzlibrary{automata, positioning, arrows}

\let\OLDthebibliography\thebibliography
\renewcommand\thebibliography[1]{
  \OLDthebibliography{#1}
  \setlength{\parskip}{0pt}
  \setlength{\itemsep}{0pt plus 0.3ex}
}

\def\peq{\preccurlyeq}

\definecolor{green1}{RGB}{153,216,201}
\definecolor{green2}{RGB}{44,190,95}

\definecolor{blue1}{RGB}{158,202,225}
\definecolor{blue2}{RGB}{49,130,189}

\definecolor{RedOrange}{cmyk}{0,0.77,0.87,0} % PANTONE 179
\definecolor{Mahogany}{cmyk}{0,0.85,0.87,0.35} % PANTONE 484
\definecolor{Maroon}{cmyk}{0,0.87,0.68,0.32} % PANTONE 201
\definecolor{BrickRed}{cmyk}{0,0.89,0.94,0.28} % PANTONE 1805
\definecolor{Red}{cmyk}{0,1.,1.,0} % PANTONE RED
\definecolor{OrangeRed}{cmyk}{0,1.,0.50,0} % No PANTONE,match

\definecolor{purple}{rgb}{0.8,0.12,0.8}
\definecolor{orange}{rgb}{1.0,0.7,0.0}
\definecolor{pink}{rgb}{1,0.5,0.8}
\definecolor{blackg}{rgb}{0.1,0.25,0.1}
\definecolor{ForestGreen}{cmyk}{0.91,0,0.88,0.42}
\definecolor{Turquoise}{cmyk}{0.85,0,0.20,0}

 \theoremstyle{plain}
 \newtheorem{thm1}{Theorem}
 \newtheorem{cor1}[thm1]{Corollary}
  
\newtheorem{thm}{Theorem}[section]
\newtheorem{lemma}[thm]{Lemma}
\newtheorem{prop}[thm]{Proposition}

\theoremstyle{definition}
\newtheorem{question}[thm]{Question}

\newtheorem{defn}[thm]{Definition}

\newtheorem{remark}[thm]{Remark}
\newtheorem{example}[thm]{Example}

\numberwithin{equation}{section}

\def\Wext{\widetilde{W}}

\def\sA{\mathsf{A}}
\def\sB{\mathsf{B}}
\def\sC{\mathsf{C}}
\def\sD{\mathsf{D}}
\def\sE{\mathsf{E}}
\def\sF{\mathsf{F}}
\def\sG{\mathsf{G}}

\def\sX{\mathsf{X}}
\def\sH{\mathsf{H}}
\def\sI{\mathsf{I}}

\def\cA{\mathcal{A}}
\def\cB{\mathcal{B}}

\def\cF{\mathcal{F}}
\def\cG{\mathcal{G}}

\def\cL{\mathcal{L}}

\def\cT{\mathcal{T}}

\def\CircFree{\mathsf{CircFree}}
\def\PrefixCircFree{\mathsf{CircFree}_{\peq}}

\def\lex{\mathsf{lex}}

\newcommand{\la}{\lambda}

\newcommand{\Ga}{\Gamma}

\def\NN{\mathbb{N}}

\def\RR{\mathbb{R}}

\def\WW{\mathbb{W}}

\def\ZZ{\mathbb{Z}}

\def\la{\lambda}

\def\<{\langle}
\def\>{\rangle}

\def\la{\lambda}

\def\sR{\mathsf{R}}

\newcommand{\sw}{\mathsf{w}}
\newcommand{\sfp}{\mathsf{p}}

\newcommand{\cLlex}{\mathcal{L}_{\mathsf{lex}}}

\newcommand{\cAlex}{\mathcal{A}_{\mathsf{lex}}}

\newcommand{\st}{\mathsf{state}}
\newcommand{\pa}{\mathsf{path}}
\newcommand{\CircWord}{\mathsf{SimpCircWord}}

\newcommand{\boundL}{\mathbf{b}_{\mathcal{L}}(\varphi)}
\newcommand{\boundG}{\mathbf{b}_{G,S}(\varphi)}
\newcommand{\cellL}{\Gamma_{\mathcal{L}}(\varphi)}
\newcommand{\cellG}{\Gamma_{G,S}(\varphi)}
\newcommand{\boundW}{\mathbf{b}_{W,S}(\varphi)}
\newcommand{\cellW}{\Gamma_{W,S}(\varphi)}

\newcommand{\ba}{\mathbf{a}}
\newcommand{\bb}{\mathbf{b}}

\newcommand{\sq}{\mathsf{q}}

\makeatletter
\renewcommand{\@makefnmark}{\mbox{\textsuperscript{}}}
\makeatother

\makeatletter
\renewcommand*{\@fnsymbol}[1]{%
  \ensuremath{%
    \ifcase#1%
    \or \dagger % First footnote symbol (was *)
    \or \ddagger % Second footnote symbol (was †)
    \or \mathsection % Third footnote symbol (was ‡)
    \or \mathparagraph % Fourth footnote symbol (was §)
    \or \| % Fifth footnote symbol (was ¶)
    \or ** % Sixth footnote symbol (was ‖)
    \or \dagger\dagger % Seventh footnote symbol (was **)
    \or \ddagger\ddagger % Eighth footnote symbol (was ††)
    \else\@ctrerr
    \fi
  }%
}
\makeatother

\title{Bounded weight functions on regular languages and groups}
\author{J. Guilhot\footnote{J. Guilhot passed away on the $27^{\text{th}}$ of July 2025}, E. Little, J. Parkinson}
\date{\today}

\begin{document}

\setlength{\textwidth}{16cm} \setlength{\textheight}{25cm}

\maketitle

\begin{abstract}
We introduce the notion of a \textit{bounded weight function} on a language, and show that the set of bounded weight functions on a regular language is a rational polyhedral cone. We study the \textit{cell} recognised by a bounded weight function (that is, the set of elements of the language where the bound is attained), and show that if the language is regular then this cell is regular. The related notion of a weight function on a finitely generated group is introduced, and the case of Coxeter groups is studied in detail. Applications to the representation theory of weighted Hecke algebras are given. 
\end{abstract}
%

%\tableofcontents

\section{Introduction}

Let $\cL$ be a language over an alphabet~$S$. A \textit{weight function} on $S^*$ is a function $\varphi:S^*\to\RR$ with $\varphi(u\cdot v)=\varphi(u)+\varphi(v)$, and we say that $\varphi$ is \textit{bounded on $\cL$} if there exists $N>0$ such that $\varphi(w)\leq N$ for all $w\in\cL$. If $\varphi$ is bounded on $\cL$ we let 
$$
\boundL=\sup_{w\in\cL}\varphi(w)\quad\text{and}\quad \cellL=\{w\in\cL\mid \varphi(w)=\boundL\}
$$
be the \textit{bound} of~$\varphi$ on $\cL$, and the \textit{cell recognised} by $\varphi$, respectively.  

The set $\WW_S$ of all weight functions on $S^*$ is an $|S|$-dimensional real vector space (identified with $\RR^{|S|}$ in an obvious way), and the set $\cB(\cL)$ of all weight functions that are bounded on~$\cL$ is a convex cone in~$\WW_S$. Recall that a cone is \textit{polyhedral} if it is the conical hull of finitely many vectors, and \textit{rational} if these vectors may be taken to be rational. Our main theorem is as follows.

\begin{thm1}\label{thm1:cones}
Let $\cL$ be a regular language. 
\begin{compactenum}[$(1)$]
\item The cone $\cB(\cL)$ of weight functions bounded on~$\cL$  is polyhedral and rational.
\item There exists a finite set $\cF\subseteq \cL$ such that $\boundL=\max\{\varphi(w)\mid w\in\cF\}$ for all $\varphi\in\cB(\cL)$. 
\item If $\varphi\in\cB(\cL)$ then the cell $\cellL$ is a nonempty regular language over~$S$. 
\end{compactenum}
\end{thm1}

Less formally, Theorem~\ref{thm1:cones} says that for a regular language, the problems of determining whether a given weight function is bounded, and computing the bound and the cell recognised by a bounded weight function, are ``finite problems''. 

The proof of Theorem~\ref{thm1:cones} is constructive, giving an explicit description of the walls of the cone $\cB(\cL)$, explicitly determining the set $\cF$, and constructing an automaton recognising~$\cellL$. If $\cL$ is not regular, then all conclusions in Theorem~\ref{thm1:cones} may fail (see Examples~\ref{ex:nonpolyhedral}, \ref{ex:nonregular}, and \ref{ex:nonrational}).

A primary motivation for studying weight functions on languages comes from the related notion of weight functions on a finitely generated group~$(G,S)$. A \textit{weight function} on $(G,S)$ is a function $\varphi:G\to \mathbb{R}$ with $\varphi(gh)=\varphi(g)+\varphi(h)$ whenever $\ell(gh)=\ell(g)+\ell(h)$ (with $\ell:G\to\NN$ the natural length function). If $\varphi:G\to\RR$ is a bounded weight function, we write
%
%
%A primary motivation for studying weight functions on languages comes from the related notion of weight functions on a finitely generated group. Let $(G,S)$ be a finitely generated group with associated length function $\ell:G\to \NN$. A \textit{weight function} on $(G,S)$ is a function $\varphi:G\to \mathbb{R}$ with $\varphi(gh)=\varphi(g)+\varphi(h)$ whenever $\ell(gh)=\ell(g)+\ell(h)$. If $\varphi:G\to\RR$ is a bounded weight function, we write
$$
\boundG=\sup_{g\in G}\varphi(g)\quad\text{and}\quad \cellG=\{g\in G\mid \varphi(g)=\boundG\}
$$
for the \textit{bound} of $\varphi$ and the \textit{cell recognised} by~$\varphi$.

\restoregeometry
\setlength{\textwidth}{16cm} \setlength{\textheight}{25.5cm}

A combinatorial interpretation of weight functions on groups is as follows. Let $\mathsf{Cay}(G,S)$ be the (directed) Cayley graph of $(G,S)$. A weight function $\varphi$ on $(G,S)$ is equivalent to choosing real labels on each directed edge of $\mathsf{Cay}(G,S)$ such that the labels are invariant under the left action of $G$ on $\mathsf{Cay}(G,S)$, and such that for all $g,h\in G$ all geodesics joining the vertex $h$ to the vertex $g$ have the same weight (where the weight of a geodesic is the sum of the edge weights along the geodesic). A weight function is bounded if and only if there is a global upper bound for the weight of geodesics in $\mathsf{Cay}(G,S)$. 

If there exists a bounded weight function $\varphi$ on $(G,S)$ with $\varphi(s)>0$ for some $s\in S$ then one can deduce numerical properties of reduced expressions in the group. For example, in the triangle group $G=\Delta(2,4,6)$ with Coxeter generating set $S=\{s,t,u\}$ (see Example~\ref{ex:246}) it turns out that the weight function with $\varphi(s)=1$, $\varphi(t)=2$, and $\varphi(u)=-5$ is bounded, with bound~$\boundG=6$. This implies that for any reduced expression $w$ in this group we have 
$
|w_s|+2|w_t|-5|w_u|\leq 6
$
where $|w_s|$ denotes the number of times the generator $s$ appears in $w$, and similarly for $|w_t|$ and $|w_u|$. Moreover, the set of elements for which equality holds is $\cellG=\{stst(utst)^n\mid n\geq 0\}$ (see Corollary~\ref{cor1:groups} below).

A \textit{geodesic language} for $(G,S)$ is a language $\cL\subseteq S^*$ consisting of reduced expressions for elements of~$(G,S)$ such that each element $g\in G$ is represented by at least one word in~$\cL$. Let $\sfp:S^*\to G$ be the natural map. Theorem~\ref{thm1:cones} gives the following corollary.

\begin{cor1}\label{cor1:groups}
Suppose that the finitely generated group $(G,S)$ admits a regular geodesic language~$\cL$. There exist finite sets $X,Y\subseteq G$, depending only on~$\cL$, such that if $\varphi:G\to \RR$ is a weight function then:
\begin{compactenum}[$(1)$]
\item $\varphi$ is bounded if and only if $\varphi(x)\leq 0$ for all $x\in X$;
%\item if $\varphi(x)<0$ for all $x\in X$ then $\Ga_{G,S}(\varphi)$ is finite;
\item if $\varphi$ is bounded then $\boundG=\max\{\varphi(y)\mid y\in Y\}$;
\item if $\varphi$ is bounded then $\{w\in \cL\mid \sfp(w)\in \cellG\}$ is a nonempty regular language over~$S$. 
\end{compactenum}
\end{cor1}

Another motivation for studying bounded weight functions comes from the representation theory of Hecke algebras with unequal parameters, and connections to Kazhdan-Lusztig theory and conjectures of Bill Casselman on the regularity of Kazhdan-Lusztig cells. We now briefly describe this motivation. Let $(W,S)$ be a Coxeter system, and let $\psi:W\to\mathbb{Z}_{\geq 0}$ be a weight function on $(W,S)$ with $\psi(s)\in\ZZ_{>0}$ for all $s\in S$. The associated \textit{weighted Hecke algebra} is the algebra $H=H(W,S,\psi)$ over $\sR=\ZZ[\sq,\sq^{-1}]$ with basis $\{T_x\mid x\in W\}$ and defining relations
$$
T_xT_s=\begin{cases}T_{xs}&\text{if $\ell(xs)=\ell(x)+1$}\\
T_{xs}+(\sq^{\psi(s)}-\sq^{-\psi(s)})T_x&\text{if $\ell(xs)=\ell(x)-1$}.
\end{cases}
$$
The ``equal parameter case'' is when $\psi=\ell$, and in this case a systematic approach to the representation theory of $H$ is given by Kazhdan and Lusztig  in the landmark paper~\cite{KL:79}. The central notion that arises is that of a \textit{Kazhdan-Lusztig cell}. These cells give a decomposition of the Coxeter group, and for each cell there is naturally associated a representation of the Hecke algebra called a \textit{cell module}. The case of ``unequal parameters'' (where $\psi:W\to\ZZ_{\geq0}$ is a general weight function) was introduced by Lusztig in~\cite{Lus:03} (see also~\cite{Bon:17}). In this case Kazhdan-Lusztig theory is much less developed and many aspects remain conjectural. %This is primarily due to the fact that the ``positivity properties'' enjoyed by the equal parameter case do not hold in general. 

In the 1990s it was conjectured by Bill Casselman that (two-sided) Kazhdan-Lusztig cells in $H$ are regular sets (see \cite{Cas:94,Cas:94b,Gun:10}). This conjecture was initially made for the equal parameter case, however it appears equally plausible in the unequal parameter case. In \cite{Gun:10} Gunnels proves Casselman's conjecture for affine Coxeter groups in the equal parameter case (making essential use of a result of Du~\cite{Du:91} that is not readily available for unequal parameters), and in \cite{BGS:14} Belolipetsky, Gunnels and Scott give conjectural descriptions of Kazhdan-Lusztig cells in hyperbolic planar Coxeter groups that, if true, would imply Casselman's conjecture for this class of groups in equal parameters. 

Recently, an approach to Kazhdan-Lusztig theory in the unequal parameter case has been developed by Guilhot and Parkinson~\cite{GP:19,GP:19b} and Chapelier-Laget, Guilhot, Little and Parkinson~\cite{CGLP:24}, directed towards proving Lusztig's conjectures $\mathsf{P1}$-$\mathsf{P15}$ and constructing Lusztig's asymptotic algebra for affine Coxeter groups. At the heart of this approach is the concept of a \textit{bounded representation} of the Hecke algebra: this is an $H$-module with a fixed choice of basis, such that the matrix entries for the action of the elements $T_x$, $x\in W$, on the module with respect to the fixed basis have globally bounded degree in~$\sq$. If $\pi$ is a bounded representation, the \textit{bound} of $\pi$ is the least upper bound $\bb(\pi)$ of the degree of the entries in the matrices $\pi(T_x)$ with $x\in W$, and the \textit{cell recognised} by the representation is the set $\Gamma(\pi)$ consisting of those $x\in W$ for which the matrix $\pi(T_x)$ attains the degree bound $\bb(\pi)$ in some entry. 

In \cite{GP:19,GP:19b} an explicit family of finite dimensional bounded representations is constructed for affine types $\tilde{\mathsf{G}}_2$ and $\tilde{\mathsf{C}}_2$ (for all choices of parameters) such that each member of the family recognises a different two-sided Kazhdan-Lusztig cell. Similarly, in \cite{CGLP:24} a family of bounded representations is constructed recognising the two-sided Kazhdan-Lusztig cells in type $\tilde{\mathsf{A}}_n$ for $n\geq 1$. In each case, the bound $\bb(\pi)$ turns out to be the value of Lusztig's $\ba$-function on the cell (see \cite[Theorem~2.6]{GP:19} for an explanation of this fact).

In light of the above discussion, and motivated by Casselman's conjecture on regularity of Kazhdan-Lusztig cells, it is natural to ask whether the cell recognised by a bounded representation of the Hecke algebra is necessarily a regular set. %Indeed, an affirmative resolution to this question would imply Casselman's conjecture, because it can be shown that the (two-sided) cell modules are bounded representations recognising the Kazhdan-Lusztig cells (assuming $\mathsf{P}1$-$\mathsf{P}15$; however note that these cell modules are typically infinite dimensional). 
Theorem~\ref{thm1:cones} resolves this question for bounded $1$-dimensional representations, giving the following Corollary.

\begin{cor1}\label{cor1:reps}
Let $\pi$ be a bounded $1$-dimensional representation of a weighted Hecke algebra. Then the language of reduced expressions for elements in $\Gamma(\pi)$ is a regular language over~$S$. 
\end{cor1}

We now outline the structure of the paper. In Section~\ref{sec:2} we give background on regular languages and prove Theorem~\ref{thm1:cones}. In Section~\ref{sec:3} we study weight functions on groups, proving Corollary~\ref{cor1:groups}. In Section~\ref{sec:4} we specialise to the case of Coxeter groups. We briefly sketch the construction of the minimal automaton recognising the language of lexicographically minimal reduced words in a Coxeter group, and provide examples where Theorem~\ref{thm1:cones} and Corollary~\ref{cor1:groups} can be made completely explicit. Moreover we give explicit formulae for the bound $\boundG$ for spherical Coxeter groups, and for the walls of the cone $\cB(\cL)$ for the language of reduced words in an affine Coxeter group. Finally, in Section~\ref{subsec:5} we prove Corollary~\ref{cor1:reps}.

\section{Proof of Theorem~\ref{thm1:cones}}\label{sec:2}

In this section we prove Theorem~\ref{thm1:cones}. 

\subsection{Languages and automata theory}

We begin with some brief background on languages and automata theory (primary references include \cite{Eps:92} and~\cite{HRR:17}). Let $S$ be a nonempty finite set, and let $S^*$ denote the free monoid over~$S$ consisting of finite strings of elements of~$S$, with $\varnothing$ denoting the empty string. The elements of $S^*$ are called \textit{words}. For $u,v\in S^*$ we write $u\cdot v\in S^*$ for concatenation of words. We say that $v\in S^*$ is a \textit{prefix} of $w\in S^*$, denoted $v\peq w$, if $w=v\cdot v'$ for some $v'\in S^*$. For $w\in S^*$ and $s\in S$ we write $|w_s|$ for the number of times the letter~$s$ appears in the word~$w$. Any subset $\cL\subseteq S^*$ is called a \textit{language} over the \textit{alphabet}~$S$.

An \textit{automaton} is a 5-tuple $\cA=(X,S,o,\tau,F)$ where
\begin{compactenum}[$\bullet$]
\item $X$ is the set of \textit{states};
\item $S$ is a finite set (the \textit{alphabet});
\item $o\in X$ is the \textit{start state};
\item $\tau:X\times S\to 2^X$ is a function (the \textit{transition function});
\item $F\subseteq X$ is the set of \textit{accept states}.
\end{compactenum}
A word $(s_1,\ldots,s_n)\in S^*$ is accepted by $\cA$ if and only if there exists a sequence of states $x_0,\ldots,x_n\in X$ such that $x_0=o$, $x_n\in F$, and $x_{i}\in \tau(x_{i-1},s_i)$ for all $1\leq i\leq n$. The \textit{language recognised by $\cA$} is the set $\mathsf{Acc}(\cA)\subseteq S^*$ of all accepted words. A language $\cL$ over~$S$ is called \textit{regular} if there exists an automaton $\cA=(X,S,o,\tau,F)$ with $|X|<\infty$ recognising $\cL$ (that is, there exists a \textit{finite state} automaton $\cA$ with $\cL=\mathsf{Acc}(\cA)$). 

It is helpful to visualise $\cA$ as a directed graph $\cG(\cA)$ with edges labelled by~$S$, as follows. The vertex set of $\cA$ is taken to be $X$, and for $x,y\in X$ there is a directed edge from $x$ to $y$ with label $s\in S$ if and only if $y\in \tau(x,s)$. Then a word $(s_1,\ldots,s_n)\in S^*$ is accepted if and only if it forms the edge labels on a path in $\cG(\cA)$ from the start state $o$ to an accept state. 

An automaton $\cA=(X,S,o,\tau,F)$ is called \textit{deterministic} if $|\tau(x,s)|\leq 1$ for all $x\in X$ and $s\in S$, and \textit{non-deterministic} otherwise. If $\cA$ is deterministic and $w\in\mathsf{Acc}(\cA)$ then there is a unique path in (the graph of) $\cA$ starting at $o$ with edge labels~$w$. Both deterministic and non-deterministic automata arise in this work, however we note that given a non-deterministic finite state automaton there is a standard procedure to construct a deterministic finite state automaton recognising the same language (see \cite[Proposition~2.5.2]{HRR:17}). 
%
%Typically this deterministic automaton has many more states than the initial non-deterministic one. In this paper we need both notions (our automaton recognising the cell $\cA_{\varphi}$ of a bounded weight function is typically non-deterministic. 
By the Myhill-Nerode Theorem \cite[Theorem~2.5.4]{HRR:17}, given a regular language $\cL$ there is a unique minimal deterministic finite state automaton recognising~$\cL$ (that is, with the fewest states, and unique up to a natural notion of equivalence).

\subsection{Weight functions on a language}

A weight function on $S^*$ is a function $\varphi:S^*\to \RR$ such that $\varphi(u\cdot v)=\varphi(u)+\varphi(v)$ for all $u,v\in S^*$ (in particular $\varphi(\varnothing)=0$). It is clear that a weight function $\varphi$ on $S^*$ is completely determined by the real numbers $a_s=\varphi(s)$, $s\in S$, and conversely every $|S|$-tuple of real numbers $(a_s)_{s\in S}$ gives rise to a weight function with $\varphi(s)=a_s$. Thus the set $\WW_S$ of all weight functions on $S^*$ is an $|S|$-dimensional real vector space. We identify $\WW_S$ with $\RR^{|S|}$ in the obvious way, via $\varphi_s\leftrightarrow e_s$ for $s\in S$, were $\varphi_s(s)=1$ and $\varphi_s(t)=0$ for $t\in S\backslash\{s\}$.

%
%Given a weight function $\varphi$, let \james{I don't think we use this, until the group case}
%\begin{align*}
%S_{\varphi}^+&=\{s\in S\mid \varphi(s)>0\},& S_{\varphi}^-&=\{s\in S\mid \varphi(s)<0\},& S_{\varphi}^0&=\{s\in S\mid \varphi(s)=0\}.
%\end{align*} 
%
%

Let $\cL$ be a language over the alphabet~$S$. A weight function $\varphi:S^*\to \RR$ is \textit{bounded on $\cL$} if there exists $N>0$ such that $\varphi(w)\leq N$ for all $w\in\cL$. Let $\cB(\cL)\subseteq\WW_S$ denote the set of all weight functions bounded on~$\cL$. If $\varphi,\varphi'\in\cB(\cL)$ then $\lambda\varphi+\lambda'\varphi'\in\cB(\cL)$ for all $\lambda,\lambda'\geq 0$, and so $\cB(\cL)$ is a convex cone in $\WW_S$. For $\varphi\in\cB(\cL)$ we define the \textit{bound} $\boundL$ of $\varphi$ and the \textit{cell} $\cellL$ recognised by $\varphi$ as in the introduction.

Before proving Theorem~\ref{thm1:cones} we give some examples of (necessarily non-regular) languages where the conclusion of Theorem~\ref{thm1:cones} fails. In particular, these examples demonstrate the following possibilities. Firstly, the cone $\cB(\cL)$ may be either non-polyhedral, or polyhedral yet non-rational. Secondly, the cell recognised by a bounded weight function may be a non-regular language. Thirdly, the bound of a bounded weight function may not be attained (and so the cell is empty). 

\begin{example}\label{ex:nonpolyhedral}
Let $S=\{s,t,u\}$ and let $\cL$ be the language over $S$ given by 
$$
\cL=\{s^{\lceil n\cos\theta\rceil}t^{\lceil n\sin\theta\rceil}u^n\mid n\in\mathbb{N},\,\theta\in[0,\pi/2]\}.
$$
Identifying $\varphi\in\WW_S$ with $(\varphi(s),\varphi(t),\varphi(u))\in\mathbb{R}^3$ we have 
$$
\cB(\cL)=\{(x,y,z)\in\mathbb{R}^3\mid \cos\theta\,x+\sin\theta\,y+z\leq 0\text{ for all $\theta\in[0,\pi/2]$}\},
$$
a non-polyhedral cone in $\mathbb{R}^3$. 
\end{example}

\begin{example}\label{ex:nonregular}
Let $S=\{s,t\}$ and let $\cL=\{s^nt^n\mid n\in\mathbb{N}\}$. The weight function $\varphi$ with $\varphi(s)=1$ and $\varphi(t)=-1$ is bounded on $\cL$, and $\cellL=\cL$, a non-regular language (for example, by the Pumping Lemma).
\end{example}

\begin{example}\label{ex:nonrational}
Let $S=\{s,t\}$ and let $\cL$ be the language over $S$ given by 
$$
\cL=\{w\in S^*\mid |w_s|-\sqrt{2}|w_t|<\sqrt{3}\}.
$$
Identifying $\varphi\in\WW_S$ with $(\varphi(s),\varphi(t))\in\mathbb{R}^2$, the cone $\cB(\cL)$ is the conical hull of the vectors $(1,-\sqrt{2})$ and $(-1,0)$, and hence is polyhedral yet non-rational. Moreover, the weight function $\varphi:S^*\to\RR$ with $\varphi(w)=|w_s|-\sqrt{2}|w_t|$ is bounded on~$\cL$, and since $\mathbb{N}-\mathbb{N}\sqrt{2}$ is dense in $\mathbb{R}$ we have $\boundL=\sqrt{3}$, yet $\varphi(w)\neq \sqrt{3}$ for any $w\in\cL$, giving $\cellL=\varnothing$. 
\end{example}

\subsection{Proof of Theorem~\ref{thm1:cones}}\label{sec:proof}

Suppose that $\cL$ is a regular language over an alphabet~$S$, and fix a deterministic finite state automaton~$\cA=(X,S,o,\tau,F)$ recognising~$\cL$. If $v\in S^*$ is a prefix of some element $w\in\cL$ then there exists a unique path $\pa(v)$ in $\cA$ starting at $o$ with edge labels~$v$, and we let $\st(v)\in X$ denote the end state of this path (note that $v\in\cL$ if and only if $\st(v)\in F$).

\begin{defn}
Let $\cL$ be a regular language recognised by the deterministic finite state automaton~$\cA$ as above. Let $w=(s_1,\ldots,s_n)\in \cL$ and let $w_i=(s_1,\ldots,s_i)\in S^*$ for $1\leq i\leq n$. 
\begin{compactenum}[$(1)$]
\item The word $w\in\cL$ is \textit{circuit free} if $\st(w_i)\neq\st(w_j)$ whenever $1\leq i<j\leq n$. Let $\CircFree(\cL)\subseteq\cL$ be the set of circuit free elements of $\cL$, and let $\PrefixCircFree(\cL)\subseteq S^*$ be the set of all prefixes of circuit free words. 
\item If $1\leq i<j\leq n$ are such that $\st(w_i)=\st(w_j)$ then the word $v=(s_{i+1},\ldots,s_j)\in S^*$ is called a \textit{circuit subword} of~$w$. Moreover if $\st(w_k)\neq \st(w_l)$ for all $i<k<l<j$ then $v$ is called a \textit{simple circuit subword} of~$w$. Let $\CircWord(\cL)\subseteq S^*$ denote the set of all simple circuit subwords, for all $w\in\cL$. 
\end{compactenum}
\end{defn}

Note that both $\CircWord(\cL)$ and $\CircFree(\cL)$ depend on the particular choice of automaton~$\cA$ recognising~$\cL$, however since we fix a choice of $\cA$ throughout, we omit this dependence from the notation. 

\begin{remark}\label{rem:terminology} The reason for the terminology is as follows. A word $w\in\cL$ is circuit free if and only if the path $p=\pa(w)$ in $\cA$ visits no state more than once (that is, it contains no circuits). On the other hand, if $v$ is a circuit subword of~$w$ then the part of the path $\pa(w)$ corresponding to the subword $v$ is a circuit based at the state $\st(w_i)$, and if $v$ is a simple circuit subword then this circuit is a simple circuit. 
\end{remark}

\begin{prop}\label{prop:finitesets}
The sets $\CircWord(\cL)$ and $\CircFree(\cL)$ are finite. 
\end{prop}

\begin{proof}
If $v\in \CircWord(\cL)$ then $v$ is the sequence of labels on a simple circuit in $\cA$ based at some state~$x$ (see Remark~\ref{rem:terminology}). Since $\cA$ has finitely many states, the number of simple circuits in~$\cA$ is finite, and hence $\CircWord(\cL)\subseteq S^*$ is finite. 

If $w\in \CircFree(\cL)$ then the path $\pa(w)$ visits no state of $\cA$ more than once, and $w$ is the sequence of edge labels on this path. Since $\cA$ is finite there are finitely many paths visiting no state more than once, and hence $\CircFree(\cL)$ is finite. 
\end{proof}

\begin{lemma}\label{lem:reducetocircfree}
%If $w\in\cL$ then there are elements $u_0,\ldots,u_r\in S^*$ and $v_1,\ldots,v_r\in \CircWord(\cL)$ with $w=u_0\cdot v_1\cdot u_1\cdot v_2\cdot u_2\cdots v_{r}\cdot u_r$ such that the word $w_{\mathsf{cf}}=u_0\cdot u_1\cdot u_2\cdots u_r$ is in $\CircFree(\cL)$ and $\st(w_{\mathsf{cf}})=\st(w)$. \james{I think this is too strong}
%%
%%$\st(u_0\cdots u_1\cdots u_r)=\st(w)$ and 
%%
%%
%%exists $w_{\mathsf{cf}}\in\CircFree(\cL)$ with $\st(w_{\mathsf{cf}})=\st(w)$ that differs from $w$ by repeated deletion of simple circuit subwords. Moreover if the deleted simple circuit subwords are $v_1,\ldots,v_k$, then $\varphi(\pi(w_{\mathsf{cf}}))=\varphi(\pi(w))-\sum_{i=1}^k\varphi(\pi(v_i))$ for all weight functions~$\varphi$. 
%%
%%\james{Older:}
%%If $w\in\cL$ then there exists $w_{\mathsf{cf}}\in\CircFree(\cL)$ with $\st(w_{\mathsf{cf}})=\st(w)$ that differs from $w$ by repeated deletion of simple circuit subwords. Moreover if the deleted simple circuit subwords are $v_1,\ldots,v_k$, then $\varphi(\pi(w_{\mathsf{cf}}))=\varphi(\pi(w))-\sum_{i=1}^k\varphi(\pi(v_i))$ for all weight functions~$\varphi$. 
Let $w\in\cL$. If $w$ is not circuit free there is $x\in \PrefixCircFree(\cL)$, $y\in S^*$ and $v\in\CircWord(\cL)$ with $w=x\cdot v\cdot y$ such that $\st(x\cdot v)=\st(x)$ and $\st(x\cdot v')\neq \st(x)$ for any prefix $v'$ of $v$ with $v'\notin\{ \varnothing,v\}$. Moreover, the word $w'=x\cdot y$ is in $\cL$ with $\st(w')=\st(w)$. 
\end{lemma}

\begin{proof}
Write $w=(s_1,\ldots,s_n)$ and let $w_i=(s_1,\ldots,s_i)$ for $1\leq i\leq n$. If $w$ is not circuit free then there is $i,j$ with $1\leq i<j\leq n$ such that $\st(w_{i})=\st(w_{j})$ with $\st(w_k)\neq \st(w_l)$ for all $i<k<l<j$, and taking $i$ minimal we have that $w_{i}$ is circuit free. Then $w=w_i\cdot v\cdot y$ where $v=(s_{i+1},\ldots,s_{j})$ is a simple circuit subword of $w$ and $y=(s_{j+1},\ldots,s_n)$. Since $\st(w_{i})=\st(w_{j})$ the word $w'=x\cdot y$ is the sequence of edge labels on a path in $\cA$ starting at $o$ and ending at $\st(w)$. Since $\st(w)$ is an accept state (as $w\in\cL$) we have $w'\in \cL$ with $\st(w')=\st(w)$. 
%
%
%
%Moreover, since the expression $w=u_0\cdot v\cdot w_j'$ is reduced we have $\varphi(\pi(w'))=\varphi(\pi(w))-\varphi(\pi(v_1))$ for all weight functions~$\varphi$. 
%
%Write $w'=(s_1',s_2',\ldots,s_m')$ and $w_i'=(s_1',\ldots,s_i')$ (and so $s_i'=s_i$ for $1\leq i\leq i_1$). Suppose that $1\leq i_2<j_2\leq m$ with $\st(w_{i_2}')=\st(w_{j_2}')$ with $\st(w_k')\neq \st(w_l')$ for all $i_2<k<l<j_2$.
%
%Continue this process until $w'$ is circuit free. 
\end{proof}
%
%\begin{remark}
%Note that the element $w_{\mathsf{cf}}$ in $\CircFree(\cL)$ in Lemma~\ref{lem:reducetocircfree} is not necessarily unique. We will fix a canonical choice of $w_{\mathsf{cf}}$ by deleting the leftmost simple circuit word first.
%\end{remark}

Theorem~\ref{thm1:cones} will follow easily from Theorem~\ref{thm:bounded} below.

\begin{thm}\label{thm:bounded}
Let $\cL$ be a regular language over the alphabet~$S$ and let $\varphi\in\WW_S$. 
\begin{compactenum}[$(1)$]
\item We have $\varphi\in\cB(\cL)$ if and only if $\varphi(v)\leq 0$ for all $v\in\CircWord(\cL)$.
\item If $\varphi(v)< 0$ for all $v\in\CircWord(\cL)$ then $\cellL\subseteq \CircFree(\cL)$ is finite. 
\item If $\varphi\in\cB(\cL)$ then $\boundL=\max\{\varphi(w) \mid w\in \CircFree(\cL)\}$.
\item If $\varphi\in\cB(\cL)$ then the cell $\cellL$ recognised by $\varphi$ is a regular language over~$S$.
\end{compactenum}
\end{thm}

\begin{proof}
Let $\cA$ be the fixed choice of deterministic finite state automata recognising~$\cL$, as above. 

(1) Suppose that $\varphi$ is bounded, and let $v\in \CircWord(\cL)$. Thus there is $w=(s_1,\ldots,s_n)\in\cL$ with $v=(s_{i+1},\ldots,s_j)$ such that $\st(w_i)=\st(w_j)$, where $w_i=(s_1,\ldots,s_i)$. Since $\st(w_i)=\st(w_j)$ the word $y=w_i\cdot v^N\cdot v'$, where $v'=(s_{j+1},\ldots,s_n)$ and $N\geq 0$, is the sequence of edge labels on a path in~$\cA$ starting at $o$ and ending at $\st(w)$, and hence is an element of~$\cL$. Since $\varphi(y)=\varphi(w_i)+N\varphi(v)+\varphi(v')$ boundedness forces $\varphi(v)\leq 0$. 

Conversely, suppose that $\varphi(v)\leq 0$ for all $v\in\CircWord(\cL)$. Let $w\in\cL$. If $w$ is not circuit free then by Lemma~\ref{lem:reducetocircfree} there is $x\in\PrefixCircFree(\cL)$, $y\in S^*$, and $v\in\CircWord(\cL)$ such that $w=x\cdot v\cdot y$ and $w'=x\cdot y\in\cL$. Then 
$$
\varphi(w)=\varphi(x)+\varphi(v)+\varphi(y)\leq \varphi(x)+\varphi(y)=\varphi(w').
$$
Repeating this process (using $w'$ in place of $w$) we eventually obtain $w_{\mathsf{cf}}\in\CircFree(\cL)$ with $\varphi(w)\leq \varphi(w_{\mathsf{cf}})$. Since $\CircFree(\cL)$ is a finite set (see Proposition~\ref{prop:finitesets}) it follows that $\varphi$ is bounded, hence (1). Moreover, if $\varphi(v)< 0$ for all $v\in\CircWord(\cL)$ and $w\notin\CircFree(\cL)$ then the inequality above gives $\varphi(w)<\varphi(w_{\mathsf{cf}})$ and so $w\notin\cellL$, hence (2).

(3) Suppose that $\varphi$ is bounded. Then by (1) we have $\varphi(v)\leq 0$ for all $v\in\CircWord(\cL)$, and the argument in the previous paragraph shows that for each $w\in\cL$ we have $\varphi(w)\leq \varphi(w_{\mathsf{cf}})$ for some $w_{\mathsf{cf}}\in\CircFree(\cL)$, and so the bound $\boundL$ is attained at an element of $\CircFree(\cL)$.

(4) Let $\varphi$ be bounded. A simple circuit $p$ in $\cA$ is called a \textit{simple $\varphi$-circuit} if the sequence of edge labels $w\in S^*$ on $p$ (defined up to cyclic shifts) satisfies $\varphi(w)=0$. 

We construct a graph $\cG_{\varphi}$ as follows (see Example~\ref{ex:infinitedihedral} for an explicit example). To begin with, each element $v$ of $\PrefixCircFree(\cL)$ is a vertex of $\cG_{\varphi}$, and there is an $s$-labelled edge from $v$ to $v'$ if and only if $v'=v\cdot s$. We add new vertices as follows. For each $v\in\PrefixCircFree(\cL)$ and each simple $\varphi$-circuit $p$ in~$\cA$ based at $\st(v)$, we append to $v$ a corresponding simple circuit of new vertices with the same edge labels as $p$. The resulting directed $S$-labelled graph $\cG_{\varphi}$ is considered as a (typically non-deterministic) automaton $\cA_{\varphi}$ with the vertex set $X_{\varphi}$ of $\cG_{\varphi}$ being the set of states, $\varnothing$ the start state, transition function given by the directed edges, and accept states $$F_{\varphi}=\{w\in \CircFree(\cL)\mid \varphi(w)=\boundL\}.$$
Since $\cA$ has finitely many simple circuits, and since $\CircFree(\cL)$ is finite, the automaton $\cA_{\varphi}$ has finitely many states.

We claim that $\mathsf{Acc}(\cA_{\varphi})=\cellL$. Suppose that $w\in \mathsf{Acc}(\cA_{\varphi})$. We must show that $w\in\cL$ and $\varphi(w)=\boundL$. Since $w$ is accepted by $\cA_{\varphi}$ there is a path $p$ in $\cA_{\varphi}$ starting at $\varnothing$ with edge labels~$w$, and the final state of this path is an element $z\in\CircFree(\cL)$ with $\varphi(z)=\boundL$. By construction of $\cA_{\varphi}$ this path decomposes as $p=p_1\cdot c_1\cdot p_2\cdot c_2\cdots  p_n\cdot c_n$ where:
\begin{compactenum}[$(a)$]
\item $p_1,p_2,\ldots,p_n$ are paths in the subgraph of $\cG_{\varphi}$ with vertex set $\PrefixCircFree(\cL)$ with the start vertex of $p_{i+1}$ being the end vertex of $p_i$ for all $1\leq i\leq n-1$, and
\item  $c_1,c_2,\ldots,c_n$ are bouquets of simple circuits with $c_i$ based at the end vertex $x_i$ of $p_1\cdot p_2\cdots p_i$ for each $1\leq i\leq n$ such that the only vertex of $c_i$ in $\PrefixCircFree(\cL)$ is~$x_i$. 
\end{compactenum}
This is illustrated below. 
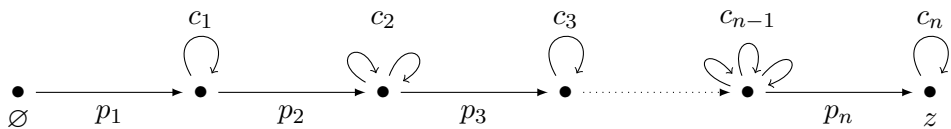
\begin{figure}[H]
\centering
\begin{tikzpicture} [xscale=1.2,yscale=1.2]
\node at (0,0) (1) {$\bullet$};
\node at (2,0) (2) {$\bullet$};
\node at (4,0) (3) {$\bullet$};
\node at (6,0) (4) {$\bullet$};
\node at (8,0) (5) {$\bullet$};
\node at (10,0) (6) {$\bullet$};
\node at (1,-0.25) {$p_1$};
\node at (3,-0.25) {$p_2$};
\node at (5,-0.25) {$p_3$};
\node at (9,-0.25) {$p_n$};
\node at (10,-0.3) {$z$};
\node at (0,-0.3) {$\varnothing$};
\node at (2,0.8) {$c_1$};
\node at (4,0.8) {$c_2$};
\node at (6,0.8) {$c_3$};
\node at (8,0.8) {$c_{n-1}$};
\node at (10,0.8) {$c_n$};
\draw[-latex] (1)--(2);
\draw[-latex] (2)--(3);
\draw[-latex] (3)--(4);
\draw[-latex,dotted] (4)--(5);
\draw[-latex] (5)--(6);
\draw[-latex] (2) edge[in=90-30,out=90+30,loop] ();
\draw[-latex] (3) edge[in=90+20,out=90+60,loop] ();
\draw[-latex] (3) edge[in=90-60,out=90-20,loop] ();
\draw[-latex] (4) edge[in=90-30,out=90+30,loop] ();
\draw[-latex] (5) edge[in=90+40,out=90+70,loop] ();
\draw[-latex] (5) edge[in=90-20,out=90+20,loop] ();
\draw[-latex] (5) edge[in=90-70,out=90-40,loop] ();
\draw[-latex] (6) edge[in=90-30,out=90+30,loop] ();
\end{tikzpicture}
\end{figure}
For each $i$ let $u_i\in S^*$ be the sequence of labels on $p_i$, and let $v_i\in S^*$ be the sequence of labels on $c_i$. By construction of $\cA_{\varphi}$ we have $\varphi(v_i)=0$ for all~$i$. The path $p_1\cdot p_2\cdots p_n$ is a path in the subgraph $\PrefixCircFree(\cL)$ with edge labels $u_1\cdot u_2\cdots u_n$, and so $z=u_1\cdot u_2\cdots u_n$. Then $\varphi(w)=\sum_i (\varphi(u_i)+\varphi(v_i))=\sum_i\varphi(u_i)=\varphi(z)=\boundL$. 

We now show that $w$ is accepted by $\cA$. Since $z\in\CircFree(\cL)$ in particular we have $z\in \cL$, and hence $z$ is accepted by $\cA$. Thus there is a path $q_0$ in $\cA$ with edge labels~$z=u_1\cdot u_2\cdots u_n$. Assume for the moment that $c_1$ is a single circuit (as illustrated above). By construction of $\cA_{\varphi}$ there is a simple $\varphi$-circuit with labels $v_1$ based at the state $\st(u_1)$ of $\cA$. If $c_1$ is a bouquet of simple circuits (as illustrated for $c_2$ above) then there is a bouquet of simple $\varphi$-circuits based at $\st(u_1)$ with combined edge labels~$v_1$. Thus there is a path $q_1$ in $\cA$ with edge labels $u_1\cdot v_1\cdot u_2\cdot u_3\cdots u_n$, and continuing inductively we obtain a path $q_n$ in $\cA$ with edge labels $u_1\cdot v_1\cdot u_2\cdot v_2\cdots u_n\cdot v_n=w$, and so $w$ is accepted by~$\cA$.

Conversely, suppose that $w\in\cellL$. Write $w=x\cdot v\cdot y$ with $x\in \PrefixCircFree(\cL)$, $y\in S^*$, and $v\in\CircWord(\cL)$, and let $w'=x\cdot y$ as in Lemma~\ref{lem:reducetocircfree}. Since $\varphi$ is bounded part (1) of the theorem gives $\varphi(v)\leq 0$ and since $w'\in\cL$ we have $\varphi(w')\leq \boundL$, and so 
$$
\boundL=\varphi(w)=\varphi(x)+\varphi(v)+\varphi(y)\leq \varphi(x)+\varphi(y)= \varphi(w')\leq \boundL.
$$
Thus $\varphi(v)=0$ and $w'\in \cellL$. We claim that if $w'$ is accepted by $\cA_{\varphi}$ then so too is $w$. To see this, note that since $x\in \PrefixCircFree(\cL)$ there is a path in $\cG_{\varphi}$ from $\varnothing$ to $x$ remaining in the subset $\PrefixCircFree(\cL)$ of the vertex set. Since $\st(x\cdot v)=\st(x)$ and $\st(x\cdot v')\neq\st(x)$ for all prefixes $v'\neq \varnothing,v$ of $v$ there is a simple circuit in $\cA$ based at $\st(x)$ with edge labels~$v$. Thus since $\varphi(v)=0$ there is, by construction, a simple circuit in $\cA_{\varphi}$ with edge labels $v$ based at the state~$x$. Assuming that $w'$ is accepted by $\cA_{\varphi}$, there is also a path in $\cA_{\varphi}$ starting at $x$ with edge labels~$y$, and thus there is a path in $\cA_{\varphi}$ with edge labels $x\cdot v\cdot y$, hence the claim. 

Continuing the process (replacing $w$ with $w'$ and decomposing $w'=x'\cdot v'\cdot y'$) we eventually arrive at a circuit free element $w_{\mathsf{cf}}\in\cL$ with $\varphi(w_{\mathsf{cf}})=\boundL$. Since this element is accepted by~$\cA_{\varphi}$, the inductive argument of the previous paragraph gives that~$w$ is also accepted by $\cA_{\varphi}$, completing the proof. 
\end{proof}
%
%\begin{thm}\label{thm:regular}
%If $\varphi$ is a bounded weight function then $\cellL$ is a regular language over~$S$.
%\end{thm}
%
%\begin{proof}
%
%\end{proof}
%

We now give the proof of Theorem~\ref{thm1:cones}.

\begin{proof}[Proof of Theorem~\ref{thm1:cones}]
(1) By Theorem~\ref{thm:bounded}(1) we have
$$
\cB(\cL)=\{\varphi\in\WW_S\mid \varphi(v)\leq 0\text{ for all $v\in \CircWord(\cL)$}\},
$$
and by Proposition~\ref{prop:finitesets} the set $\CircWord(\cL)$ is finite. We have $\varphi(v)=\sum_{s\in S}|v_s|\varphi(s)$, where $|v_s|$ denotes the number of times the letter~$s$ appears in $v$, and so $\cB(\cL)$ is described by finitely many inequalities of the form $\sum_{s\in S}a_s\varphi(s)\leq 0$, where $a_s\in\mathbb{Z}_{\geq 0}$. Thus by the Minkowski-Weyl Theorem (see \cite[Theorem~1.3]{Zie:98}) $\cB(\cL)$ is a rational polyhedral cone. 

(2) This follows from Theorem~\ref{thm:bounded}(3), noting that $\CircFree(\cL)$ is finite by Proposition~\ref{prop:finitesets}. 

(3) This is Theorem~\ref{thm:bounded}(4). 
\end{proof}

\begin{example}\label{ex:infinitedihedral}
To illustrate Theorem~\ref{thm:bounded} (and hence Theorem~\ref{thm1:cones}) in a simple example, let $S=\{s,t\}$ and let $\cL$ be the language accepted by the following automata~$\cA$ (with start state $0$, all states accept states, and black, red arrows indicating $s$, $t$ transitions respectively).
\begin{figure}[H]
\centering
\begin{tikzpicture} [xscale=1.4,yscale=1.4]
\node[shape=circle,draw,rounded corners,fill=lightgray] at (0,0) (e) {$0$};
\node[shape=circle,draw,rounded corners,fill=lightgray] at (-1,0) (t) {$2$};
\node[shape=circle,draw,rounded corners,fill=lightgray] at (1,0) (s) {$1$};
\draw[-latex,thick] (e)--(s);
\draw[-latex,red,thick] (e)--(t);
\draw[-latex,thick,red] (s) to[out=135,in=45] (t);
\draw[-latex,thick] (t) to[out=-45,in=-135] (s);
\end{tikzpicture}
\end{figure}
\noindent Then $\CircFree(\cL)=\{\varnothing,(s),(t),(s,t),(t,s)\}$ and $\CircWord(\cL)=\{(s,t),(t,s)\}$. Let $\varphi$ be the weight function with $\varphi(s)=a$ and $\varphi(t)=b$ for $a,b\in\RR$. By Theorem~\ref{thm:bounded} we have that $\varphi$ is bounded if and only if $a+b\leq 0$, and if $\varphi$ is bounded then 
$$
\boundL=\max\{\varphi(w)\mid w\in\CircFree(\cL)\}=\max\{0,a,b,a+b\}.
$$
To illustrate the construction of the automaton $\cA_{\varphi}$ from the proof of Theorem~\ref{thm:bounded}, consider the case $a>0$ with $a+b=0$. The automaton $\cA_{\varphi}$ is illustrated on the left in Figure~\ref{fig:A1} (the start state is~$\varnothing$, and there is only one accept state $s$, shaded grey). We have $\PrefixCircFree(\cL)= \CircFree(\cL)$, and the nodes labelled $\varnothing,s,t,st,ts$ correspond to these elements. The nodes with $\bullet$ are those circuits added corresponding to the simple $\varphi$-circuits in $\cA$. 
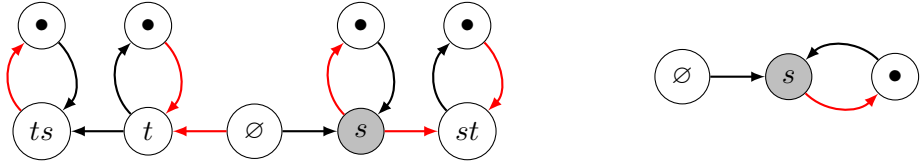
\begin{figure}[H]
\centering
\begin{tikzpicture} [xscale=1.4,yscale=1.4]
\node[shape=circle,draw,rounded corners] at (0,0) (e) {\small{$\varnothing$}};
\node[shape=circle,draw,rounded corners] at (-1,0) (t) {$t$};
\node[shape=circle,draw,rounded corners] at (-2,0) (ts) {$ts$};
\node[shape=circle,draw,rounded corners,fill=lightgray] at (1,0) (s) {$s$};
\node[shape=circle,draw,rounded corners] at (2,0) (st) {$st$};
\node[shape=circle,draw,rounded corners] at (-2,1) (newts) {$\bullet$};
\node[shape=circle,draw,rounded corners] at (-1,1) (newt) {$\bullet$};
\node[shape=circle,draw,rounded corners] at (1,1) (news) {$\bullet$};
\node[shape=circle,draw,rounded corners] at (2,1) (newst) {$\bullet$};
\draw[-latex,thick] (e)--(s);
\draw[-latex,red,thick] (s)--(st);
\draw[-latex,red,thick] (e)--(t);
\draw[-latex,thick] (t)--(ts);
\draw[-latex,red,thick] (ts) to[out=-225,in=-135] (newts);
\draw[-latex,thick] (newts) to[out=-45,in=45] (ts);
\draw[-latex,thick] (t) to[out=135,in=225] (newt);
\draw[-latex,red,thick] (newt) to[out=-45,in=45] (t);
\draw[-latex,thick] (st) to[out=-225,in=-135] (newst);
\draw[-latex,red,thick] (newst) to[out=-45,in=45] (st);
\draw[-latex,red,thick] (s) to[out=135,in=225] (news);
\draw[-latex,thick] (news) to[out=-45,in=45] (s);
\end{tikzpicture}\qquad\quad\qquad 
\begin{tikzpicture} [xscale=1.4,yscale=1.4]
\node at (0,-0.7) {};
\node[shape=circle,draw,rounded corners] at (0,0) (e) {\small{$\varnothing$}};
\node[shape=circle,draw,rounded corners,fill=lightgray] at (1,0) (s) {$s$};
\node[shape=circle,draw,rounded corners] at (2,0) (news) {$\bullet$};
\draw[-latex,thick] (e)--(s);
\draw[-latex,red,thick] (s) to[out=-45,in=225] (news);
\draw[-latex,thick] (news) to[out=90+45,in=45] (s);
\end{tikzpicture}
\caption{Automata recognising $\cellL$}\label{fig:A1}
\end{figure}
\noindent There is some redundancy in $\cA_{\varphi}$, as any path passing through $t$ or $st$ can never result in an accept state. Thus the reduced automaton on the right in Figure~\ref{fig:A1} recognises the same language~$\cellL$. Explicitly we have $\cellL=\{(s),(s,t,s),(s,t,s,t,s),\ldots\}$. 
%\begin{figure}[H]
%\centering
%\begin{tikzpicture} [xscale=1.4,yscale=1.4]
%\node[shape=circle,draw,rounded corners] at (0,0) (e) {$\varnothing$};
%\node[shape=circle,draw,rounded corners,fill=lightgray] at (1,0) (s) {$s$};
%\node[shape=circle,draw,rounded corners] at (2,0) (news) {$\bullet$};
%\draw[-latex,thick] (e)--(s);
%\draw[-latex,red,thick] (s) to[out=-45,in=225] (news);
%\draw[-latex,thick] (news) to[out=90+45,in=45] (s);
%\end{tikzpicture}
%\end{figure}
\end{example}

\begin{example}\label{ex:333a}
Let $S=\{s,t,u\}$ and consider the language $\cL$ over $S$ accepted by the automata in Figure~\ref{fig:333} (with start state $0$, all states accept states, and black, blue, red arrows indicating $s$, $t$, $u$ transitions respectively). 
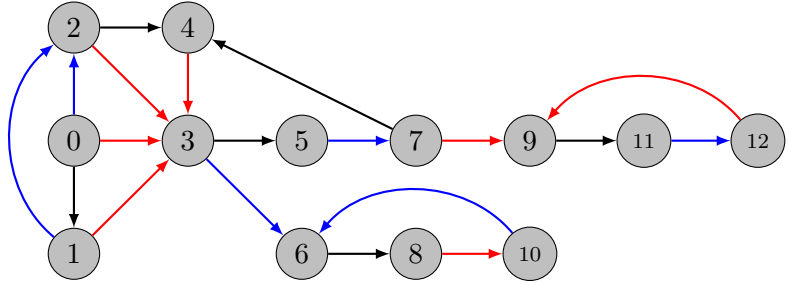
\begin{figure}[H]
\centering
\begin{tikzpicture} [xscale=1.5,yscale=1.5]
\node[shape=circle,draw,rounded corners,fill=lightgray] at (0,0) (0) {$0$};
\node[shape=circle,draw,rounded corners,fill=lightgray] at (1,0) (3) {$3$};
\node[shape=circle,draw,rounded corners,fill=lightgray] at (2,0) (5) {$5$};
\node[shape=circle,draw,rounded corners,fill=lightgray] at (3,0) (7) {$7$};
\node[shape=circle,draw,rounded corners,fill=lightgray] at (4,0) (9) {$9$};
\node[shape=circle,draw,rounded corners,fill=lightgray] at (5,0) (11) {\scriptsize{$11$}};
\node[shape=circle,draw,rounded corners,fill=lightgray] at (6,0) (12) {\scriptsize{$12$}};
\node[shape=circle,draw,rounded corners,fill=lightgray] at (0,1) (2) {$2$};
\node[shape=circle,draw,rounded corners,fill=lightgray] at (0,-1) (1) {$1$};
\node[shape=circle,draw,rounded corners,fill=lightgray] at (1,1) (4) {$4$};
\node[shape=circle,draw,rounded corners,fill=lightgray] at (2,-1) (6) {$6$};
\node[shape=circle,draw,rounded corners,fill=lightgray] at (3,-1) (8) {$8$};
\node[shape=circle,draw,rounded corners,fill=lightgray] at (4,-1) (10) {\scriptsize{$10$}};
\draw[-latex,black,thick] (0)--(1);
\draw[-latex,blue,thick] (0)--(2);
\draw[-latex,red,thick] (0)--(3);
\draw[-latex,blue,thick] (1) to[out=90+50,in=270-50] (2);
\draw[-latex,red,thick] (2)--(3);
\draw[-latex,red,thick] (4)--(3);
\draw[-latex,red,thick] (1)--(3);
\draw[-latex,red,thick] (8)--(10);
\draw[-latex,red,thick] (7)--(9);
\draw[-latex,red,thick] (12) to[out=90+40,in=90-40] (9);
\draw[-latex,black,thick] (2)--(4);
\draw[-latex,black,thick] (7)--(4);
\draw[-latex,black,thick] (3)--(5);
\draw[-latex,black,thick] (6)--(8);
\draw[-latex,black,thick] (9)--(11);
\draw[-latex,blue,thick] (3)--(6);
\draw[-latex,blue,thick] (5)--(7);
\draw[-latex,blue,thick] (11)--(12);
\draw[-latex,blue,thick] (10) to[out=90+40,in=90-40] (6);
\end{tikzpicture}
\caption{The automaton $\cAlex$ for the triangle group $\Delta(3,3,3)$ (see Example~\ref{ex:333b})}\label{fig:333}
\end{figure}
\noindent The language~$\cL$ turns out to be the language of lexicographically minimal reduced words in the Coxeter group $W=\langle s,t,u\mid s^2=t^2=u^2=(st)^3=(su)^3=(tu)^3=1\rangle$ of type $\tilde{\mathsf{A}}_2$, with lexicographic order $s<t<u$ (see Section~\ref{sec:geodesic}). The set $\CircWord(\cL)$ consists of the words $(s,t,u)$, $(s,t,s,u)$ and all cyclic permutations of these words, and so a weight function $\varphi$ with $\varphi(s)=a$, $\varphi(t)=b$, and $\varphi(u)=c$ is bounded if and only if $a+b+c\leq 0$ and $2a+b+c\leq 0$. It follows that the cone $\cB(\cL)$ is generated by the bounded weight functions $\varphi_1,\varphi_2,\varphi_3,\varphi_4$ with 
 \begin{align*}
(\varphi_1(s),\varphi_1(t),\varphi_1(u))&=(0,1,-1) &(\varphi_2(s),\varphi_2(t),\varphi_2(u))&=(0,-1,1)\\
 (\varphi_3(s),\varphi_3(t),\varphi_3(u))&=(-1,0,1)&(\varphi_4(s),\varphi_4(t),\varphi_4(u))&=(1,0,-2).
 \end{align*}
 Note that $\cB(\cL)$ is not strictly convex in this example (as it contains the line generated by $\varphi_1$). 

There are $64$ elements in $\CircFree(\cL)$, and it follows that 
$
\boundL=\max\{
0,
a,
b,
c,
2a+c,
2a+b,
2b+c
\}.
$
For example, if $(a,b,c)=(1,-1,-1)$ then $\varphi$ is bounded, and $\boundL=1$. 
\end{example}

\section{Weight functions on groups}\label{sec:3}

In this section we consider the related notion of weight functions on finitely generated groups. 

\subsection{Definitions}\label{sec:definitions}

Let $G$ be a group generated by a finite set~$S$ (we do not assume that $S=S^{-1}$). Let $\ell:G\to \NN$ denote the associated \textit{length function}, where 
$$\ell(g)=\min\{k\geq 0\mid g=s_1\cdots s_k\text{ with }s_1,\ldots,s_k\in S\},$$ 
and so $d(g,h)=\ell(g^{-1}h)$ is the graph distance from $g$ to $h$ in the directed Cayley graph of $(G,S)$. Let $\sfp:S^*\to G$ be the natural map $\sfp(s_1,s_2,\ldots)=s_1s_2\cdots$. A word $w=(s_1,\ldots, s_k)\in S^*$ is \textit{reduced} if $\ell(\sfp(w))=k$. Let $\cL(G,S)\subseteq S^*$ denote the language of all reduced words in $(G,S)$. 

\begin{defn}
A \textit{weight function on $(G,S)$} is a function $\varphi:G\to\RR$ with $\varphi(gh)=\varphi(g)+\varphi(h)$ whenever $\ell(gh)=\ell(g)+\ell(h)$. A weight function $\varphi$ on $G$ is \textit{bounded} if there exists $N>0$ such that $\varphi(g)\leq N$ for all $g\in G$. If $\varphi$ is bounded we define the \textit{bound} $\bb_{G,S}(\varphi)$ of $\varphi$ and the \textit{cell} $\Ga_{G,S}(\varphi)$ recognised by~$\varphi$ as in the introduction. 
\end{defn}

It is clear that the set $\WW_{G,S}$ of all weight functions on $(G,S)$ is a real vector space, and since a weight function $\varphi$ on $(G,S)$ is completely determined by the real numbers $\varphi(s)$, $s\in S$, we have $\dim \WW_{G,S}\leq |S|$. However, unlike the case of weight functions on languages, not all choices of real numbers $a_s$, $s\in S$, necessarily extend to a weight function~$\varphi$ with $\varphi(s)=a_s$ (as illustrated in the following example), and so $\dim\WW_{G,S}$ may be strictly smaller than~$|S|$.  Since the map $\varphi(g)=\lambda \ell(g)$ is always a weight function for all $\lambda\in\RR$, we have $\dim\WW_{G,S}\geq 1$. Note also that the set $\cB(G,S)$ of all bounded weight functions on $(G,S)$ is a convex cone in~$\WW_{G,S}$. 

\begin{example}
The set of weight functions on~$(G,S)$ depends on the generating set~$S$ chosen. The modular group $G=\mathsf{PSL}_2(\ZZ)$ has the following presentations: 
$$G=\langle s,t\mid s^2=t^3=1\rangle=\langle s,u\mid s^2=(su)^3=1\rangle.$$ 
Taking $S=\{s,t\}$ each element $g\in G$ has a unique reduced expression (indeed $G\cong C_2*C_3$), and thus each assignment $\varphi(s)=a$ and $\varphi(t)=b$ with $a,b\in\RR$ extends to a weight function on $(G,S)$. Moreover the explicit form of reduced expressions implies that $\varphi$ is bounded if and only if $a+b\leq 0$ and $a+2b\leq 0$. If $S'=\{s,t,t^{-1}\}$ we again have unique reduced expressions, and this time each assignment $\varphi(s)=a$, $\varphi(t)=b$, and $\varphi(t^{-1})=c$ with $a,b,c\in\RR$ extends to a weight function on $(G,S')$, with $\varphi$ bounded if and only if $a+b\leq 0$ and $a+c\leq 0$. Finally, taking $S''=\{s,u,u^{-1}\}$ we have the relations $sus=u^{-1}su^{-1}$ and $usu=su^{-1}s$ (with each expression reduced) and it follows that every weight function $\varphi$ on $(G,S'')$ has $\varphi(s)=\varphi(u)=\varphi(u^{-1})$. Thus the only weight functions on $(G,S'')$ are the functions $\varphi(g)=\lambda \ell(g)$ for some $\lambda\in\mathbb{R}$, and such a weight function is bounded if and only if $\la\leq 0$.
\end{example}

The connection between weight functions on $(G,S)$ and weight functions on $S^*$ is given in the following proposition. A weight function $\varphi_0:S^*\to \RR$ is \textit{compatible} with $(G,S)$ if $\varphi_0(w)=\varphi_0(w')$ whenever $w,w'\in \cL(G,S)$ with $\sfp(w)=\sfp(w')$. 

\begin{prop}\label{prop:weightfunctionconnection}
A function $\varphi:G\to \RR$ is a weight function on $(G,S)$ if and only if there exists a weight function $\varphi_0:S^*\to\RR$ compatible with $(G,S)$ such that $\varphi_0(w)=\varphi(\sfp(w))$ for all $w\in\cL(G,S)$. 
\end{prop}

\begin{proof}
Let $\varphi:G\to\RR$ be a weight function on $(G,S)$. If $w=(s_1,\ldots,s_n)\in\cL(G,S)$ with $\sfp(w)=g$ then $\varphi(g)=\varphi(s_1\cdots s_n)=\varphi(s_1)+\cdots+\varphi(s_n)=\varphi_0(w)$ where $\varphi_0:S^*\to \RR$ is the weight function with $\varphi_0(s)=\varphi(s)$ for all $s\in S$, and it follows that $\varphi_0$ is compatible with $(G,S)$. Conversely, if $\varphi_0:S^*\to \RR$ is compatible with $(G,S)$ then setting $\varphi(g)=\varphi_0(w)$ for any $w\in\cL(G,S)$ with $\sfp(w)=g$ is well defined, and if $g=g_1g_2$ with $\ell(g_1g_2)=\ell(g_1)+\ell(g_2)$ then choosing $w_1,w_2\in\cL(G,S)$ with $\sfp(w_1)=g_1$ and $\sfp(w_2)=g_2$ we have $\sfp(w_1\cdot w_2)=g$ and $\varphi(g)=\varphi_0(w_1\cdot w_2)=\varphi_0(w_1)+\varphi(w_2)=\varphi(g_1)+\varphi(g_2)$ and so $\varphi$ is a weight function on~$(G,S)$.  
\end{proof}

In other words, Proposition~\ref{prop:weightfunctionconnection} says that $\WW_{G,S}$ is naturally isomorphic to the subspace of $\WW_S$ consisting of all weight functions on $S^*$ compatible with~$(G,S)$. 

\subsection{Geodesic languages}\label{sec:geodesic}

\begin{defn}
A \textit{geodesic language} for $(G,S)$ is a language $\cL\subseteq \cL(G,S)$ such that $\sfp:\cL\to G$ is surjective. A geodesic language $\cL$ is \textit{exact} if $\sfp:\cL\to G$ is a bijection (that is, each element of $G$ is represented by a unique word in~$\cL$). 
\end{defn}

For example, $\cL(G,S)$ is a geodesic language, and the lexicographically minimal reduced expressions $\cLlex(G,S)$, defined below, is an exact geodesic language. To define $\cLlex(G,S)$, fix an arbitrary total order $\leq$ on $S$ and extend to the \textit{shortlex} total order $\leq$ on~$S^*$ (with words ordered by length, with words of the same length ordered lexicographically). For each $g\in G$ there exists a unique word $\mathsf{lex}(g)\in \cL(G,S)$ representing $g$ minimal in the shortlex order, and we set 
$$
\cLlex(G,S)=\{\mathsf{lex}(g)\mid g\in G\}.
$$

We are particularly interested in the case where $(G,S)$ admits a geodesic language that is regular. For example if $(G,S)$ admits a geodesic automatic structure in the sense of~\cite[\S5.3]{HRR:17} then there exists a regular geodesic language for~$(G,S)$ (however note that by~\cite[Section~3.5]{Eps:92} not all automatic groups admit a geodesic automatic structure for a fixed generating set). 

We now prove Corollary~\ref{cor1:groups}.

\begin{proof}[Proof of Corollary~\ref{cor1:groups}]
Taking $X=\{\sfp(w)\mid w\in\CircWord(\cL)\}$ and $Y=\{\sfp(w)\mid w\in\CircFree(\cL)\}$ the corollary follows immediately from Theorem~\ref{thm:bounded} and Proposition~\ref{prop:weightfunctionconnection}. 
\end{proof}

\section{Coxeter groups}\label{sec:4}

In this section we specialise to the case of Coxeter groups. Let $S$ be a finite set, and 
$$
W=\langle S\mid (st)^{m_{s,t}}=1\text{ for $s,t\in S$}\rangle,
$$
where $m_{s,s}=1$ for all $s\in S$, and $m_{s,t}=m_{t,s}\in\ZZ_{\geq 2}\cup\{\infty\}$ for $s,t\in S$ with $s\neq t$ (if $m_{s,t}=\infty$ it is understood that the relation $(st)^{m_{s,t}}=1$ is omitted). 

A Coxeter system is \textit{irreducible} if there exists no nontrivial partition $S=S_1\sqcup S_2$ with $m_{s,t}=2$ for all $s\in S_1$ and $t\in S_2$. If $(W,S)$ is \textit{spherical} (that is, $|W|<\infty$) there exists a unique element $\sw_0\in W$ of maximal length. For $J\subseteq S$ we write $W_J=\langle J\rangle$ for the associated parabolic subgroup, and if $(W_J,J)$ is spherical write $\sw_J$ for the longest element of~$W_J$.

It was shown by Brink and Howlett~\cite{BH:93} that both $\cLlex(W,S)$ and $\cL(W,S)$ are regular languages. There are various explicit constructions of automata recognising these languages (see \cite{BH:93,Cas:94,Cas:94b} for $\cLlex(W,S)$, and~\cite{HNW:16,PY:19,PY:22,HN:23} for $\cL(W,S)$). For computational purposes it is usually more efficient to work with $\cLlex=\cLlex(W,S)$ as it typically admits an automaton with a smaller number of states, and for our purposes this leads to smaller sets $\CircFree(\cLlex)$ and $\CircWord(\cLlex)$.

In this section we make Theorem~\ref{thm1:cones} explicit for various classes of Coxeter groups (via Theorem~\ref{thm:bounded} and Corollary~\ref{cor1:groups}).

\subsection{The minimal automata recognising $\cLlex(W,S)$}

Let $\cAlex=\cAlex(W,S)$ denote the minimal deterministic automaton recognising~$\cLlex=\cLlex(W,S)$. We describe $\cAlex$ briefly below (following~\cite{Cas:94,Cas:94b}). The \textit{shortlex cone type} of $x\in W$ is 
$$
T(x)=\{y\in W\mid \lex(xy)=\lex(x)\cdot \lex(y)\}.
$$
Let $\cT=\{T(x)\mid x\in W\}$ denote the set of all shortlex cone types. Since $\cLlex$ is regular, the proof of the Myhill-Nerode Theorem (see \cite[Theorem~2.5.4]{HRR:17} and \cite[Theorem~1.23]{PY:22}) implies that $|\cT|<\infty$.

The following proposition (with elementary proof omitted) gives an inductive procedure for calculation of $T(x)$ (see \cite[Appendix]{GMP:18}, \cite{HN:23}, and \cite{PY:22} for details on cone types in the language $\cL(W,S)$). 

\begin{prop}\label{prop:conetype}
Let $s\in S$ and $T\in\cT$. 
\begin{compactenum}[$(1)$]
\item We have $T(s)=\{x\in W\mid \ell(sx)=\ell(x)+1\text{ and }\ell(tsx)=\ell(x)+2\text{ for all $t\in S$ with $t< s$}\}$.
\item Suppose that $s\in T$. Then $T'=sT\cap T(s)$ is a cone type, and for all $x\in W$ with $T=T(x)$ we have $\lex(xs)=\lex(x)\cdot s$ and $T'=T(xs)$.  
\end{compactenum}
\end{prop}

\begin{thm}\label{thm:minimalshortlex}
We have $\cAlex=(\cT,S,T(e),\tau,\cT)$ where for $T\in\cT$ and $s\in S$ the transition function $\tau(T,s)$ is defined if and only if $s\in T$, and in this case $
\tau(T,s)=sT\cap T(s)$. 
\end{thm}

\begin{proof}
This follows from Proposition~\ref{prop:conetype} and the standard proof of the Myhill-Nerode Theorem (see \cite[Theorem~1.23]{PY:22} for a proof in a similar context). 
\end{proof}

\subsection{Weight functions on Coxeter groups}

It is easy to describe the set of all weight functions on a Coxeter group. 

\begin{prop}[{see \cite[\S3]{Lus:03}}]\label{prop:wf}
Let $(W,S)$ be a Coxeter system, and let $(a_s)_{s\in S}$ be an $|S|$-tuple of real numbers. There exists a weight function $\varphi:W\to \RR$ with $\varphi(s)=a_s$ if and only if $a_s=a_t$ whenever $m_{s,t}$ is odd. 
\end{prop}

For example, in a Coxeter group of type $\tilde{\sC}_n$ with $n\geq 2$ (with standard Bourbaki~\cite{Bou:02} labelling) the weight functions are given by the choices $\varphi(s_0)=a$, $\varphi(s_1)=\varphi(s_2)=\cdots=\varphi(s_{n-1})=b$ and $\varphi(s_n)=c$, with $a,b,c\in\RR$, while for a simply laced irreducible Coxeter system the only weight functions are the functions $\varphi(g)=\lambda\ell(g)$ with $\lambda\in\mathbb{R}$.

If $\varphi$ is a weight function on $(W,S)$ we write 
\begin{align*}
S_{\varphi}^+&=\{s\in S\mid \varphi(s)>0\},& S_{\varphi}^-&=\{s\in S\mid \varphi(s)<0\},& S_{\varphi}^0&=\{s\in S\mid \varphi(s)=0\},
\end{align*} 
and let $W_{\varphi}^+=\langle S_{\varphi}^+\rangle$ and similarly for $W_{\varphi}^-$ and $W_{\varphi}^0$. In particular, note that these subgroups are standard parabolic subgroups of $(W,S)$. 

\begin{prop}\label{prop:basic1}
If $\varphi$ is a bounded weight function on a Coxeter system~$(W,S)$ then every element of $\cellW$ is both a maximal length $(W_{\varphi}^+,W_{\varphi}^+)$-double coset representative, and a minimal length $(W_{\varphi}^-,W_{\varphi}^-)$-double coset representative, and the cell $\cellW$ is a union of $(W_{\varphi}^0,W_{\varphi}^0)$-double cosets. 
\end{prop}

\begin{proof}
Let $g\in \cellW$. If $s\in S_{\varphi}^+$ with $\ell(sg)=\ell(g)+1$ then $\varphi(sg)=\varphi(g)+\varphi(s)>\varphi(g)=\boundW$, a contradiction. Thus $\ell(sg)=\ell(g)-1$ and similarly $\ell(gs)=\ell(g)-1$ for all $s\in S_{\varphi}^+$, showing that $g$ is maximal length in its $(W_{\varphi}^+,W_{\varphi}^+)$-double coset. The remaining statements are similar.   
\end{proof}

\subsection{Triangle groups}

The \textit{triangle group} $W=\Delta(p,q,r)$ is the Coxeter group with generating set $S=\{s,t,u\}$ and presentation
$$
\Delta(p,q,r)=\langle s,t,u\mid s^2=t^2=u^2=(su)^p=(st)^q=(tu)^r=1\rangle,
$$
where $p,q,r\geq 2$. This group is spherical (respectively affine, hyperbolic) if $p^{-1}+q^{-1}+r^{-1}$ is greater than $1$ (respectively equal to $1$, less than $1$). While it would be possible to make Theorem~\ref{thm1:cones} explicit for all triangle groups, the analysis would be rather complicated with case distinctions and parity considerations. Therefore we content ourselves here to some illustrative examples. 

\begin{example}[The group $\Delta(2,4,6)$]\label{ex:246}

Consider the hyperbolic triangle group $\Delta(2,4,6)$. 
By Proposition~\ref{prop:wf} each assignment $\varphi(s)=a$, $\varphi(t)=b$, and $\varphi(u)=c$ with $a,b,c\in\RR$ extends to a weight function.

The minimal shortlex automaton $\cAlex$ (using the order $s<t<u$) can be computed using Proposition~\ref{prop:conetype} and Theorem~\ref{thm:minimalshortlex} (or using Derek Holt's KBMAG algorithms~\cite{HRR:17} which are implemented in $\mathsf{MAGMA}$~\cite{MAGMA}). The automaton is illustrated in Figure~\ref{fig:246}, where black, blue, and red arrows indicate $s$, $t$, and $u$ transitions respectively, $0$ is the start state, and all states are accept states. 
\begin{figure}[H]
\centering
\begin{tikzpicture} [xscale=1.5,yscale=1.1]
\node[shape=circle,draw,rounded corners,fill=lightgray] at (-0.4,0) (e) {$0$};
\node[shape=circle,draw,rounded corners,fill=lightgray] at (-1.5,0) (1) {$1$};
\node[shape=circle,draw,rounded corners,fill=lightgray] at (0.5,1) (2) {$2$};
\node[shape=circle,draw,rounded corners,fill=lightgray] at (1.75,1) (21) {\scriptsize{$10$}};
\node[shape=circle,draw,rounded corners,fill=lightgray] at (3,1) (212) {\scriptsize{$11$}};
\node[shape=circle,draw,rounded corners,fill=lightgray] at (4.5,1) (3232132) {$9$};
\node[shape=circle,draw,rounded corners,fill=lightgray] at (6,1) (323213) {$8$};
\node[shape=circle,draw,rounded corners,fill=lightgray] at (0.5,-1) (3) {$3$};
\node[shape=circle,draw,rounded corners,fill=lightgray] at (1.75,-1) (32) {$4$};
\node[shape=circle,draw,rounded corners,fill=lightgray] at (3,-1) (323) {$5$};
\node[shape=circle,draw,rounded corners,fill=lightgray] at (4.5,-1) (3232) {$6$};
\node[shape=circle,draw,rounded corners,fill=lightgray] at (6,-1) (32321) {$7$};
\node[shape=circle,draw,rounded corners,fill=lightgray] at (4.5,-2) (32323) {\scriptsize{$12$}};
\draw[-latex,thick] (e)--(1);
\draw[-latex,blue,thick] (e)--(2);
\draw[-latex,red,thick] (e)--(3);
\draw[-latex,blue,thick] (1)--(2);
\draw[-latex,red,thick] (1)--(3);
\draw[-latex,red,thick] (2)--(3);
\draw[-latex,thick] (2)--(21);
\draw[-latex,blue,thick] (3)--(32);
\draw[-latex,red,thick] (21)--(3);
\draw[-latex,blue,thick] (21)--(212);
\draw[-latex,thick] (32)--(21);
\draw[-latex,red,thick] (32)--(323);
\draw[-latex,red,thick] (212)--(3);
\draw[-latex,blue,thick] (323)--(3232);
\draw[-latex,red,thick] (3232)--(32323);
\draw[-latex,thick] (3232)--(32321);
\draw[-latex,blue,thick] (32321)--(212);
\draw[-latex,red,thick] (32321)--(323213);
\draw[-latex,blue,thick] (323213)--(3232132);
\draw[-latex,thick] (3232132)--(212);
\draw[-latex,red,thick] (3232132)--(323);
\end{tikzpicture}
\caption{The automata $\cAlex$ for the triangle group $\Delta(2,4,6)$}\label{fig:246}
\end{figure}
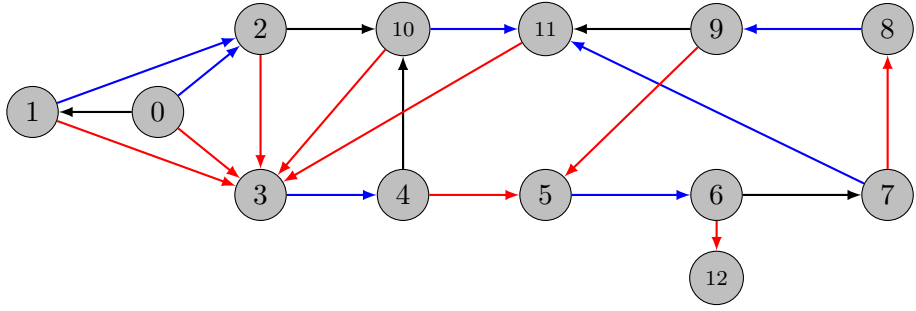

%
%
%\begin{remark}
%For $w\in W$ let $U(w)=\{v\in W\mid \ell(wv)=\ell(w)+\ell(v)\}$ and let $\cU=\{U(w)\mid w\in W\}$ (this is the set of all cone types for the regular language $\cL(W,S)$ of all reduced words). In \cite[Theorem~1]{PY:22} it is shown that for each $U\in \cU$ the set $\{w\in W\mid U(w)=U\}$ has a unique minimal length element $m_U$, and moreover if $w\in W$ with $U(w)=U$ then $m_U$ is a suffix of $w$. The automaton in Figure~\ref{fig:IIII} shows that the corresponding statement for the shortlex language is false: in particular, taking $T=T(212)$ (state 11) the minimal length element of $\{w\in W\mid T(w)\}$ is $m_T=212$, however it is not true that $\overline{m_T}=(2,1,2)$ is a suffix of $\overline{w}$ whenever $T(w)=T$ (consider $\overline{w}=(3,2,3,2,1,3,2,1)$). 
%\end{remark}
%

By inspection of Figure~\ref{fig:246} there are $5$ simple circuits in~$\cAlex$, given by $3\to 4\to 10\to 3$, $3\to 4\to 10\to 11\to 3$, $3\to 4\to 5\to 6\to 7\to 11\to 3$, $3\to 4\to 5\to 6\to 7\to 8\to 9\to 11\to 3$, and $5\to 6\to 7\to 8\to 9\to 5$. Thus by Theorem~\ref{thm:bounded} the weight function $\varphi$ is bounded if and only if 
$$
a+b+c\leq 0,\quad a+2b+c\leq 0,\quad a+3b+2c\leq 0,\quad\text{and}\quad  a+2b+2c\leq 0
$$ (note that the inequality $2a+3b+3c\leq 0$ arising from the simple circuit $3\to 4\to 5\to 6\to 7\to 8\to 9\to 11\to 3$ is redundant). It follows that the rational cone $\cB(\cLlex)$ is generated by the bounded weight functions $\varphi_1,\varphi_2,\varphi_3,\varphi_4$ with
\begin{align*}
(\varphi_1(s),\varphi_1(t),\varphi_1(u))&=(1,0,-1) &(\varphi_2(s),\varphi_2(t),\varphi_2(u))&=(0,-1,1)\\
 (\varphi_3(s),\varphi_3(t),\varphi_3(u))&=(-1,1,-1)&(\varphi_4(s),\varphi_4(t),\varphi_4(u))&=(-2,0,1)
 \end{align*}
 (unlike Example~\ref{ex:333a}, here the cone $\cB(\cLlex)$ is strictly convex). 
 %to compute this, solve each set of two equations; then check if one of the rays is in the cone. 

By Theorem~\ref{thm:bounded} the bound of a bounded weight function is attained on $\CircFree(\cLlex)$. Determining this set is a straightforward but somewhat tedious exercise (there are $92$ circuit free words). By directly considering the values of a bounded weight function $\varphi$ on this finite set of words (and making use of the inequalities $a+b+c\leq 0$, $a+2b+c\leq 0$, $a+3b+2c\leq 0$, and $a+2b+2c\leq 0$) we see that 
$$
\boundW=\max\{0,a,b,c,a+2b,2a+b,2a+2b,2b+3c,3b+2c,3b+3c,a+c\}.
$$

%
%
%\begin{prop}\label{prop:246}
%A weight function $\varphi$ on $\Delta(2,4,6)$ is bounded if and only if $a+b+c\leq 0$, $a+2b+c\leq 0$, $a+3b+2c\leq 0$, and $a+2b+2c\leq 0$. If $\varphi$ is bounded then we have 
%$$
%\boundW=\max\{0,a,b,c,a+2b,2a+b,2a+2b,2b+3c,3b+2c,3b+3c,a+c\}.
%$$
%\end{prop}
%
%\begin{proof}
%By inspection of Figure~\ref{fig:246} there are $5$ simple circuits in~$\cAlex$, given by $3\to 4\to 10\to 3$, $3\to 4\to 10\to 11\to 3$, $3\to 4\to 5\to 6\to 7\to 11\to 3$, $3\to 4\to 5\to 6\to 7\to 8\to 9\to 11\to 3$, and $5\to 6\to 7\to 8\to 9\to 5$. Thus by Theorem~\ref{thm:bounded} the weight function $\varphi$ is bounded if and only if $a+b+c\leq 0$, $a+2b+c\leq 0$,  $a+3b+2c\leq 0$, $2a+3b+3c\leq 0$, and $a+2b+2c\leq 0$. Note that $2a+3b+3c\leq 0$ is redundant. 
%
%Determining the set $\CircFree(\cLlex)$ is a straightforward but somewhat tedious exercise (there are $92$ circuit free words). By directly considering the values of a bounded weight function $\varphi$ on this finite set (and making use of the inequalities $a+b+c\leq 0$, $a+2b+c\leq 0$, $a+3b+2c\leq 0$, and $a+2b+2c\leq 0$) we obtain the formula for~$\boundW$
%\end{proof}

In particular, the weight function $\varphi=\varphi_3$ listed above is bounded, with $\boundW=1$. In this case two of the above inequalities determining boundedness become equalities: $a+2b+c=0$ and $a+3b+2c=0$. This gives rise to two simple $\varphi$-circuits in $\cAlex$, given by $3\mapsto 4\mapsto 10\mapsto 11\mapsto 3$ and $3\mapsto 4\mapsto 5\mapsto 6\mapsto 7\mapsto 11\mapsto 3$. The subset of words $w\in \CircFree(\cLlex)$ with $\varphi(\sfp(w))=\boundW=1$ is
$
\cL_1=\{t,tst,tut,tstut,tutut,tstutut,tututst\}
$
(these give the accept states of the automaton $\cA_{\varphi}$ constructed in the proof of Theorem~\ref{thm:bounded}). From the construction of $\cA_{\varphi}$ it is clear that we can ignore the states from $\CircFree(\cLlex)$ that are not prefixes of one of the above accept states. Therefore we must consider the the following subset of $\CircFree(\cLlex)$:
$$
X_0=\{e,t,ts,tu,tst,tut,tstu,tstut,tutu,tutut,tstutu,tstutut,tututs,tututst\}.
$$
After appending circuits to each states $x\in X_0$ such that there is a simple $\varphi$-circuit in $\cAlex$ based at the corresponding state $\st(x)$ of~$\cAlex$, we arrive at the automaton~$\cA_{\varphi}$ illustrated in Figure~\ref{fig:246a}. The labelled states correspond to the subset $X_0$ of $\CircFree(\cLlex)$, and we have used the notation $\substack{x\\ \st(x)}$ to denote the states (where $x\in X_0$ and $\st(x)$ is the corresponding state of $\cAlex$). The accept states are shaded.

\begin{figure}[H]
\centering
\begin{tikzpicture} [xscale=2,yscale=1.35]
\node[shape=rectangle,draw,rounded corners] at (0.65,0) (e) {$\substack{e\\0}$};
\node[shape=rectangle,draw,rounded corners,fill=lightgray]  at (1.3,0) (2) {$\substack{t\\2}$};
\node[shape=rectangle,draw,rounded corners] at (2,2-.2) (21) {$\substack{ts\\10}$};
\node[shape=rectangle,draw,rounded corners,fill=lightgray]  at (3,2-.2) (212) {$\substack{tst\\11}$};
\node[shape=rectangle,draw,rounded corners] at (4,2-.2) (2123) {$\substack{tstu\\3}$};
\node[shape=rectangle,draw,rounded corners,fill=lightgray]  at (5,2-.2) (21232) {$\substack{tstut\\4}$};
\node[shape=rectangle,draw,rounded corners] at (6,2-.2) (212323) {$\substack{tstutu\\5}$};
\node[shape=rectangle,draw,rounded corners,fill=lightgray]  at (7,2-.2) (2123232) {$\substack{tstutut\\6}$};
\node[shape=rectangle,draw,rounded corners] at (2,-2+.2) (23) {$\substack{tu\\3}$};
\node[shape=rectangle,draw,rounded corners,fill=lightgray]  at (3,-2+.2) (232) {$\substack{tut\\4}$};
\node[shape=rectangle,draw,rounded corners] at (4,-2+.2) (2323) {$\substack{tutu\\5}$};
\node[shape=rectangle,draw,rounded corners,fill=lightgray]  at (5,-2+.2) (23232) {$\substack{tutut\\6}$};
\node[shape=rectangle,draw,rounded corners] at (6,-2+.2) (232321) {$\substack{tututs\\7}$};
\node[shape=rectangle,draw,rounded corners,fill=lightgray]  at (7,-2+.2) (2323212) {$\substack{tututst\\11}$};
\node[shape=rectangle,draw,rounded corners] at (2-0.35,-2+0.55+.2) (a11) {\tiny{$\bullet$}};
\node[shape=rectangle,draw,rounded corners] at (2-0.35,-2+1.1+.2) (a12) {\tiny{$\bullet$}};
\node[shape=rectangle,draw,rounded corners] at (2,-2+1.45+.2) (a13) {\tiny{$\bullet$}};
\node[shape=rectangle,draw,rounded corners] at (2+0.35,-2+1.1+.2) (a14) {\tiny{$\bullet$}};
\node[shape=rectangle,draw,rounded corners] at (2+0.35,-2+0.55+.2) (a15) {\tiny{$\bullet$}};
\node[shape=rectangle,draw,rounded corners] at (3-0.35,-2+0.55+.2) (a21) {\tiny{$\bullet$}};
\node[shape=rectangle,draw,rounded corners] at (3-0.35,-2+1.1+.2) (a22) {\tiny{$\bullet$}};
\node[shape=rectangle,draw,rounded corners] at (3,-2+1.45+.2) (a23) {\tiny{$\bullet$}};
\node[shape=rectangle,draw,rounded corners] at (3+0.35,-2+1.1+.2) (a24) {\tiny{$\bullet$}};
\node[shape=rectangle,draw,rounded corners] at (3+0.35,-2+0.55+.2) (a25) {\tiny{$\bullet$}};
\node[shape=rectangle,draw,rounded corners] at (4-0.35,-2+0.55+.2) (a31) {\tiny{$\bullet$}};
\node[shape=rectangle,draw,rounded corners] at (4-0.35,-2+1.1+.2) (a32) {\tiny{$\bullet$}};
\node[shape=rectangle,draw,rounded corners] at (4,-2+1.45+.2) (a33) {\tiny{$\bullet$}};
\node[shape=rectangle,draw,rounded corners] at (4+0.35,-2+1.1+.2) (a34) {\tiny{$\bullet$}};
\node[shape=rectangle,draw,rounded corners] at (4+0.35,-2+0.55+.2) (a35) {\tiny{$\bullet$}};
\node[shape=rectangle,draw,rounded corners] at (5-0.35,-2+0.55+.2) (a41) {\tiny{$\bullet$}};
\node[shape=rectangle,draw,rounded corners] at (5-0.35,-2+1.1+.2) (a42) {\tiny{$\bullet$}};
\node[shape=rectangle,draw,rounded corners] at (5,-2+1.45+.2) (a43) {\tiny{$\bullet$}};
\node[shape=rectangle,draw,rounded corners] at (5+0.35,-2+1.1+.2) (a44) {\tiny{$\bullet$}};
\node[shape=rectangle,draw,rounded corners] at (5+0.35,-2+0.55+.2) (a45) {\tiny{$\bullet$}};
\node[shape=rectangle,draw,rounded corners] at (6-0.35,-2+0.55+.2) (a51) {\tiny{$\bullet$}};
\node[shape=rectangle,draw,rounded corners] at (6-0.35,-2+1.1+.2) (a52) {\tiny{$\bullet$}};
\node[shape=rectangle,draw,rounded corners] at (6,-2+1.45+.2) (a53) {\tiny{$\bullet$}};
\node[shape=rectangle,draw,rounded corners] at (6+0.35,-2+1.1+.2) (a54) {\tiny{$\bullet$}};
\node[shape=rectangle,draw,rounded corners] at (6+0.35,-2+0.55+.2) (a55) {\tiny{$\bullet$}};
\node[shape=rectangle,draw,rounded corners] at (7-0.35,-2+0.55+.2) (a61) {\tiny{$\bullet$}};
\node[shape=rectangle,draw,rounded corners] at (7-0.35,-2+1.1+.2) (a62) {\tiny{$\bullet$}};
\node[shape=rectangle,draw,rounded corners] at (7,-2+1.45+.2) (a63) {\tiny{$\bullet$}};
\node[shape=rectangle,draw,rounded corners] at (7+0.35,-2+1.1+.2) (a64) {\tiny{$\bullet$}};
\node[shape=rectangle,draw,rounded corners] at (7+0.35,-2+0.55+.2) (a65) {\tiny{$\bullet$}};
\node[shape=rectangle,draw,rounded corners] at (3-0.35,2-0.55-.2) (b21) {\tiny{$\bullet$}};
\node[shape=rectangle,draw,rounded corners] at (3-0.35,2-1.1-.2) (b22) {\tiny{$\bullet$}};
\node[shape=rectangle,draw,rounded corners] at (3,2-1.45-.2) (b23) {\tiny{$\bullet$}};
\node[shape=rectangle,draw,rounded corners] at (3+0.35,2-1.1-.2) (b24) {\tiny{$\bullet$}};
\node[shape=rectangle,draw,rounded corners] at (3+0.35,2-0.55-.2) (b25) {\tiny{$\bullet$}};
\node[shape=rectangle,draw,rounded corners] at (4-0.35,2-0.55-.2) (b31) {\tiny{$\bullet$}};
\node[shape=rectangle,draw,rounded corners] at (4-0.35,2-1.1-.2) (b32) {\tiny{$\bullet$}};
\node[shape=rectangle,draw,rounded corners] at (4,2-1.45-.2) (b33) {\tiny{$\bullet$}};
\node[shape=rectangle,draw,rounded corners] at (4+0.35,2-1.1-.2) (b34) {\tiny{$\bullet$}};
\node[shape=rectangle,draw,rounded corners] at (4+0.35,2-0.55-.2) (b35) {\tiny{$\bullet$}};
\node[shape=rectangle,draw,rounded corners] at (5-0.35,2-0.55-.2) (b41) {\tiny{$\bullet$}};
\node[shape=rectangle,draw,rounded corners] at (5-0.35,2-1.1-.2) (b42) {\tiny{$\bullet$}};
\node[shape=rectangle,draw,rounded corners] at (5,2-1.45-.2) (b43) {\tiny{$\bullet$}};
\node[shape=rectangle,draw,rounded corners] at (5+0.35,2-1.1-.2) (b44) {\tiny{$\bullet$}};
\node[shape=rectangle,draw,rounded corners] at (5+0.35,2-0.55-.2) (b45) {\tiny{$\bullet$}};
\node[shape=rectangle,draw,rounded corners] at (6-0.35,2-0.55-.2) (b51) {\tiny{$\bullet$}};
\node[shape=rectangle,draw,rounded corners] at (6-0.35,2-1.1-.2) (b52) {\tiny{$\bullet$}};
\node[shape=rectangle,draw,rounded corners] at (6,2-1.45-.2) (b53) {\tiny{$\bullet$}};
\node[shape=rectangle,draw,rounded corners] at (6+0.35,2-1.1-.2) (b54) {\tiny{$\bullet$}};
\node[shape=rectangle,draw,rounded corners] at (6+0.35,2-0.55-.2) (b55) {\tiny{$\bullet$}};
\node[shape=rectangle,draw,rounded corners] at (7-0.35,2-0.55-.2) (b61) {\tiny{$\bullet$}};
\node[shape=rectangle,draw,rounded corners] at (7-0.35,2-1.1-.2) (b62) {\tiny{$\bullet$}};
\node[shape=rectangle,draw,rounded corners] at (7,2-1.45-.2) (b63) {\tiny{$\bullet$}};
\node[shape=rectangle,draw,rounded corners] at (7+0.35,2-1.1-.2) (b64) {\tiny{$\bullet$}};
\node[shape=rectangle,draw,rounded corners] at (7+0.35,2-0.55-.2) (b65) {\tiny{$\bullet$}};
\node[shape=rectangle,draw,rounded corners] at (2-0.25,-2-0.55+.2) (c11) {\tiny{$\bullet$}};
\node[shape=rectangle,draw,rounded corners] at (2,-2-1+.2) (c12) {\tiny{$\bullet$}};
\node[shape=rectangle,draw,rounded corners] at (2+0.25,-2-0.55+.2) (c13) {\tiny{$\bullet$}};
\node[shape=rectangle,draw,rounded corners] at (3-0.25,-2-0.55+.2) (c21) {\tiny{$\bullet$}};
\node[shape=rectangle,draw,rounded corners] at (3,-2-1+.2) (c22) {\tiny{$\bullet$}};
\node[shape=rectangle,draw,rounded corners] at (3+0.25,-2-0.55+.2) (c23) {\tiny{$\bullet$}};
\node[shape=rectangle,draw,rounded corners] at (7-0.25,-2-0.55+.2) (c31) {\tiny{$\bullet$}};
\node[shape=rectangle,draw,rounded corners] at (7,-2-1+.2) (c32) {\tiny{$\bullet$}};
\node[shape=rectangle,draw,rounded corners] at (7+0.25,-2-0.55+.2) (c33) {\tiny{$\bullet$}};
\node[shape=rectangle,draw,rounded corners] at (2-0.25,2+0.55-.2) (d11) {\tiny{$\bullet$}};
\node[shape=rectangle,draw,rounded corners] at (2,2+1-.2) (d12) {\tiny{$\bullet$}};
\node[shape=rectangle,draw,rounded corners] at (2+0.25,2+0.55-.2) (d13) {\tiny{$\bullet$}};
\node[shape=rectangle,draw,rounded corners] at (3-0.25,2+0.55-.2) (d21) {\tiny{$\bullet$}};
\node[shape=rectangle,draw,rounded corners] at (3,2+1-.2) (d22) {\tiny{$\bullet$}};
\node[shape=rectangle,draw,rounded corners] at (3+0.25,2+0.55-.2) (d23) {\tiny{$\bullet$}};
\node[shape=rectangle,draw,rounded corners] at (4-0.25,2+0.55-.2) (d31) {\tiny{$\bullet$}};
\node[shape=rectangle,draw,rounded corners] at (4,2+1-.2) (d32) {\tiny{$\bullet$}};
\node[shape=rectangle,draw,rounded corners] at (4+0.25,2+0.55-.2) (d33) {\tiny{$\bullet$}};
\node[shape=rectangle,draw,rounded corners] at (5-0.25,2+0.55-.2) (d41) {\tiny{$\bullet$}};
\node[shape=rectangle,draw,rounded corners] at (5,2+1-.2) (d42) {\tiny{$\bullet$}};
\node[shape=rectangle,draw,rounded corners] at (5+0.25,2+0.55-.2) (d43) {\tiny{$\bullet$}};
\draw[-latex,thick,blue] (e)--(2);
\draw[-latex,thick] (2) to[out=90,in=180] (21);
\draw[-latex,thick,red] (2) to[out=-90,in=180] (23);
\draw[-latex,thick,blue] (21)--(212);
\draw[-latex,thick,red] (212)--(2123);
\draw[-latex,thick,blue] (2123)--(21232);
\draw[-latex,thick,red] (21232)--(212323);
\draw[-latex,thick,blue] (212323)--(2123232);
\draw[-latex,thick,blue] (23)--(232);
\draw[-latex,thick,red] (232)--(2323);
\draw[-latex,thick,blue] (2323)--(23232);
\draw[-latex,thick] (23232)--(232321);
\draw[-latex,thick,blue] (232321)--(2323212);
\draw[-latex,thick,blue] (23)--(a11);
\draw[-latex,thick,red] (a11)--(a12);
\draw[-latex,thick,blue] (a12)--(a13);
\draw[-latex,thick] (a13)--(a14);
\draw[-latex,thick,blue] (a14)--(a15);
\draw[-latex,thick,red] (a15)--(23);
\draw[-latex,thick,red] (232)--(a21);
\draw[-latex,thick,blue] (a21)--(a22);
\draw[-latex,thick] (a22)--(a23);
\draw[-latex,thick,blue] (a23)--(a24);
\draw[-latex,thick,red] (a24)--(a25);
\draw[-latex,thick,blue] (a25)--(232);
\draw[-latex,thick,blue] (2323)--(a31);
\draw[-latex,thick,black] (a31)--(a32);
\draw[-latex,thick,blue] (a32)--(a33);
\draw[-latex,thick,red] (a33)--(a34);
\draw[-latex,thick,blue] (a34)--(a35);
\draw[-latex,thick,red] (a35)--(2323);
\draw[-latex,thick,black] (23232)--(a41);
\draw[-latex,thick,blue] (a41)--(a42);
\draw[-latex,thick,red] (a42)--(a43);
\draw[-latex,thick,blue] (a43)--(a44);
\draw[-latex,thick,red] (a44)--(a45);
\draw[-latex,thick,blue] (a45)--(23232);
\draw[-latex,thick,blue] (232321)--(a51);
\draw[-latex,thick,red] (a51)--(a52);
\draw[-latex,thick,blue] (a52)--(a53);
\draw[-latex,thick,red] (a53)--(a54);
\draw[-latex,thick,blue] (a54)--(a55);
\draw[-latex,thick,black] (a55)--(232321);
\draw[-latex,thick,red] (2323212)--(a61);
\draw[-latex,thick,blue] (a61)--(a62);
\draw[-latex,thick,red] (a62)--(a63);
\draw[-latex,thick,blue] (a63)--(a64);
\draw[-latex,thick,black] (a64)--(a65);
\draw[-latex,thick,blue] (a65)--(2323212);
\draw[-latex,thick,red] (212)--(b21);
\draw[-latex,thick,blue] (b21)--(b22);
\draw[-latex,thick,red] (b22)--(b23);
\draw[-latex,thick,blue] (b23)--(b24);
\draw[-latex,thick,black] (b24)--(b25);
\draw[-latex,thick,blue] (b25)--(212);
\draw[-latex,thick,blue] (2123)--(b31);
\draw[-latex,thick,red] (b31)--(b32);
\draw[-latex,thick,blue] (b32)--(b33);
\draw[-latex,thick,black] (b33)--(b34);
\draw[-latex,thick,blue] (b34)--(b35);
\draw[-latex,thick,red] (b35)--(2123);
\draw[-latex,thick,red] (21232)--(b41);
\draw[-latex,thick,blue] (b41)--(b42);
\draw[-latex,thick,black] (b42)--(b43);
\draw[-latex,thick,blue] (b43)--(b44);
\draw[-latex,thick,red] (b44)--(b45);
\draw[-latex,thick,blue] (b45)--(21232);
\draw[-latex,thick,blue] (212323)--(b51);
\draw[-latex,thick,black] (b51)--(b52);
\draw[-latex,thick,blue] (b52)--(b53);
\draw[-latex,thick,red] (b53)--(b54);
\draw[-latex,thick,blue] (b54)--(b55);
\draw[-latex,thick,red] (b55)--(212323);
\draw[-latex,thick,black] (2123232)--(b61);
\draw[-latex,thick,blue] (b61)--(b62);
\draw[-latex,thick,red] (b62)--(b63);
\draw[-latex,thick,blue] (b63)--(b64);
\draw[-latex,thick,red] (b64)--(b65);
\draw[-latex,thick,blue] (b65)--(2123232);
\draw[-latex,thick,blue] (23)--(c11);
\draw[-latex,thick,black] (c11)--(c12);
\draw[-latex,thick,blue] (c12)--(c13);
\draw[-latex,thick,red] (c13)--(23);
\draw[-latex,thick,black] (232)--(c21);
\draw[-latex,thick,blue] (c21)--(c22);
\draw[-latex,thick,red] (c22)--(c23);
\draw[-latex,thick,blue] (c23)--(232);
\draw[-latex,thick,red] (2323212)--(c31);
\draw[-latex,thick,blue] (c31)--(c32);
\draw[-latex,thick,black] (c32)--(c33);
\draw[-latex,thick,blue] (c33)--(2323212);
\draw[-latex,thick,blue] (21)--(d11);
\draw[-latex,thick,red] (d11)--(d12);
\draw[-latex,thick,blue] (d12)--(d13);
\draw[-latex,thick,black] (d13)--(21);
\draw[-latex,thick,red] (212)--(d21);
\draw[-latex,thick,blue] (d21)--(d22);
\draw[-latex,thick,black] (d22)--(d23);
\draw[-latex,thick,blue] (d23)--(212);
\draw[-latex,thick,blue] (2123)--(d31);
\draw[-latex,thick,black] (d31)--(d32);
\draw[-latex,thick,blue] (d32)--(d33);
\draw[-latex,thick,red] (d33)--(2123);
\draw[-latex,thick,black] (21232)--(d41);
\draw[-latex,thick,blue] (d41)--(d42);
\draw[-latex,thick,red] (d42)--(d43);
\draw[-latex,thick,blue] (d43)--(21232);
\end{tikzpicture}
\caption{An automaton recognising $\Gamma_{W,S}(\varphi_3)$}\label{fig:246a}
\end{figure}
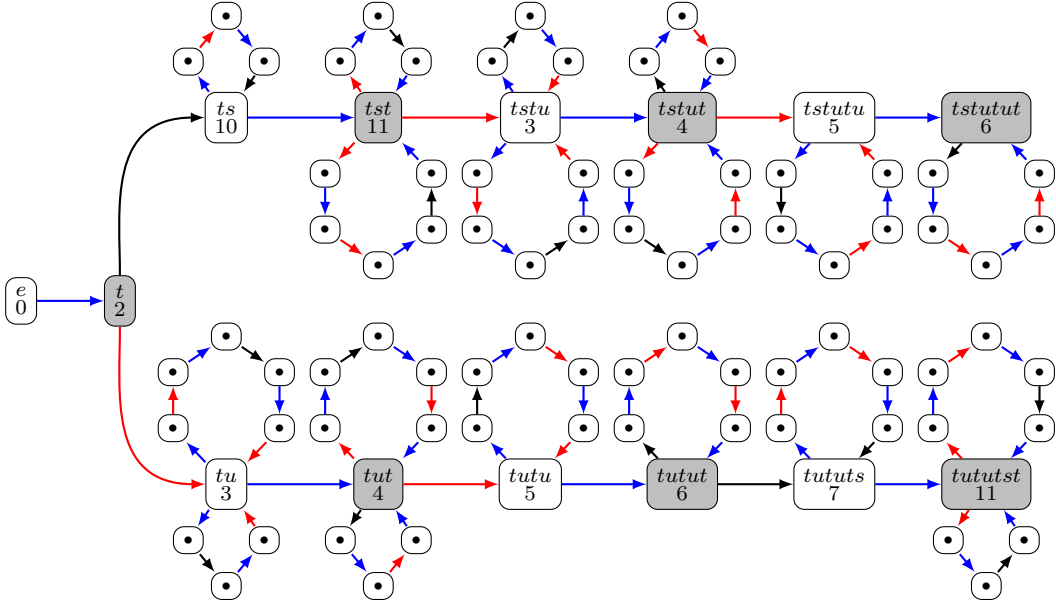

\noindent Some simplifications are possible in the automaton (c.f. Example~\ref{ex:infinitedihedral}). For example, the circuits based at the non-accept states can be removed (without altering the accepted language). 
\end{example}

\begin{example}[The group $\Delta(2,3,2m)$]\label{ex:232m}
Let $3\leq m<\infty$ and consider the triangle group $\Delta(2,3,2m)$. By Proposition~\ref{prop:wf} the weight functions on $\Delta(2,3,2m)$ are determined by the choices $\varphi(s)=\varphi(t)=a$ and $\varphi(u)=b$ with $a,b\in\RR$. The minimal shortlex automaton $\cAlex$ (using the order $s<t<u$) can be computed using Proposition~\ref{prop:conetype} and Theorem~\ref{thm:minimalshortlex}. The automaton is illustrated in Figure~\ref{fig:232m}, where black, blue, and red arrows indicate $s$, $t$, and $u$ transitions respectively, $0$ is the start state, and all states are accept states (the states are denoted $0,1,\ldots,2m+1,x,y,z,w$, and in this example we do not encircle the states for typesetting reasons). 
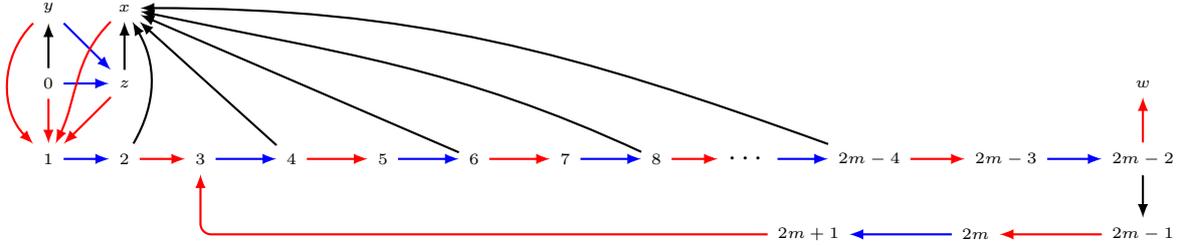
\begin{figure}[H]
\centering
\begin{tikzpicture} [scale=1]
\node at (0,0) (e) {\tiny{$0$}};
\node at (0,1) (1) {\tiny{$y$}};
\node at (1,0) (2) {\tiny{$z$}};
\node at (1,1) (21) {\tiny{$x$}};
\node at (0,-1) (3) {\tiny{$1$}};
\node at (1,-1) (32) {\tiny{$2$}};
\node at (2,-1) (323) {\tiny{$3$}};
\node at (3.2,-1) (3232) {\tiny{$4$}};
\node at (4.4,-1) (32323) {\tiny{$5$}};
\node at (5.6,-1) (323232) {\tiny{$6$}};
\node at (6.8,-1) (3232323) {\tiny{$7$}};
\node at (8,-1) (32323232) {\tiny{$8$}};
%\node[shape=rectangle,draw,rounded corners] at (9.2,-1) (y) {$\cdots$};
\node at (9.2,-1) (y) {$\cdots$};
\node at (10.8,-1) (xp) {\tiny{$2m-4$}};
\node at (12.6,-1) (xp3) {\tiny{$2m-3$}};
\node at (14.4,-1) (x) {\tiny{$2m-2$}};
\node at (14.4,0) (x3) {\tiny{$w$}};
\node at (14.4,-2) (x1) {\tiny{$2m-1$}};
\node at (12.2,-2) (x13) {\tiny{$2m$}};
\node at (10,-2) (x132) {\tiny{$2m+1$}};
\draw[-latex,thick] (e)--(1);
\draw[-latex,thick,blue]%,dashed] 
(e)--(2);
\draw[-latex,thick,red]%,dotted] 
(e)--(3);
\draw[-latex,thick,blue]%,dashed] 
(1)--(2);
\draw[-latex,thick,red]%,dotted] 
(1) to[out=225,in=135] (3);
\draw[-latex,thick] (2)--(21);
\draw[-latex,thick,red]%,dotted] 
(2)--(3);
\draw[-latex,thick,red]%,dotted] 
(21) to[out=-135,in=65] (3);
\draw[-latex,thick,blue]%,dashed] 
(3)--(32);
\draw[-latex,thick,red]%,dotted] 
(32)--(323);
\draw[-latex,thick,blue]%,dashed] 
(323)--(3232);
\draw[-latex,thick,red]%,dotted] 
(3232)--(32323);
\draw[-latex,thick,blue]%,dashed] 
(32323)--(323232);
\draw[-latex,thick,red]%,dotted] 
(323232)--(3232323);
\draw[-latex,thick,blue]%,dashed] 
(3232323)--(32323232);
\draw[-latex,thick,red]%,dotted] 
(32323232)--(y);
\draw[-latex,thick,blue]%,dashed] 
(y)--(xp);
\draw[-latex,thick,red]%,dotted] 
(xp)--(xp3);
\draw[-latex,thick,blue]%,dashed] 
(xp3)--(x);
\draw[-latex,thick,red]%,dotted] 
(x)--(x3);
\draw[-latex,thick] (x)--(x1);
\draw[-latex,thick,red]%,dotted] 
(x1)--(x13);
\draw[-latex,thick,blue]%,dashed] 
(x13)--(x132);
\draw[-latex,thick,red,rounded corners]%,dotted,rounded corners] 
(x132)--(2,-2)--(323);
\draw[-latex,thick,thick] (32) to[out=60,in=-60] (21);
%\draw[-latex] (3232) to[out=135,in=-45] (21);
%\draw[-latex] (323232) to[out=150,in=-25] (21);
\draw[-latex,thick] (32323232) to[out=155,in=-12] (21);
\draw[-latex,thick] (xp) to[out=160,in=0] (21);
\draw[-latex,thick] (3232)--(21);
\draw[-latex,thick] (323232)--(21);
%\draw[-latex] (32323232)--(21);
%\draw[-latex] (xp)--(21);
\end{tikzpicture}
\caption{The minimal shortlex automata for $\Delta(2,3,2m)$, $m\geq 3$}\label{fig:232m}
\end{figure}

\noindent By inspection the simple circuits are 
$
1\to 2\to 3\to\cdots \to 2j\to x\to 1
$ (with $2\leq 2j\leq 2m-4$) and $3\to 4\to\cdots\to 2m-1\to 2m\to 2m+1\to 3$, and it follows from Corollary~\ref{cor1:groups} that $\varphi$ is bounded if and only if $(i+1)a+ib\leq 0$ for all $1\leq i\leq m-1$. The inequalities with $1<i<m-1$ are redundant, and therefore $\varphi$ is bounded if and only if
$$2a+b\leq 0\quad\text{and}\quad ma+(m-1)b\leq 0.
$$ 
By inspection, $\CircFree(\cLlex)$ consists of the words of the form
\begin{align*}
&\{\varnothing,s,t,st,sts\}\cdot [ut]_j&&\text{with $0\leq j\leq 2m-1$, or}\\
&\{\varnothing,s,t,st\}\cdot [ut]_{2i}\cdot s&&\text{with $1\leq i\leq m-2$, or}\\
&\{\varnothing,s,t,st,sts\}\cdot [ut]_{2m-2}\cdot\{s,su,sut\}
\end{align*}
where $[ut]_j=utut\cdots$ with $j$ terms. By directly considering the values of a bounded weight function $\varphi$ on this finite set we obtain
$
\boundW=\max\{0,a,2a,3a,a+b,ia+(i+1)b\mid 0\leq i\leq m-1\}.
$
Since $2a+b\leq 0$ and $ma+(m-1)b\leq 0$ there is some redundancy, and it follows that
$$
\boundW=\max\{0,3a,b,(m-1)a+mb\}.
$$
An example of calculating $\cA_{\varphi}$ is given in Example~\ref{ex:246}, and we omit such an example here. 
\end{example}

\begin{example}\label{ex:333b}
Consider the triangle group $W=\Delta(3,3,3)$ (an affine Coxeter group of type~$\tilde{\mathsf{A}}_2$). The automaton $\cAlex$ is given in Example~\ref{ex:333a}. By Proposition~\ref{prop:wf} the only weight functions on $W$ are the functions $\varphi(g)=\lambda \ell(g)$ (and such a function is bounded if and only if $\lambda\leq 0$). Therefore Corollary~\ref{cor1:groups} gives no information. However we note that working with weight functions on the exact geodesic language~$\cLlex$ (rather than weight functions on the group) gives richer information as weight functions then have $3$ degrees of freedom. For example, following from Example~\ref{ex:333a}, the weight function $\varphi:S^*\to \RR$ with $\varphi(s)=0$, $\varphi(t)=1$ and $\varphi(u)=-1$ is bounded with bound~$1$. It follows that for every $w\in\cLlex$ we have
$
|w_t|\leq 1+|w_u|
$
(note that this bound does not hold in the language $\cL(W,S)$ of all reduced expressions, for example consider $tutst$).
\end{example}

A striking feature of Examples~\ref{ex:246} and~\ref{ex:232m} is that for the triangle groups $\Delta(2,3,2m)$ and $\Delta(2,4,6)$ the bound $\boundW$ of a bounded weight function is always attained on an element of a spherical parabolic subgroup. We ask the following question.

\begin{question}\label{q:spherical}
Let $\varphi$ be a bounded weight function on a Coxeter system $(W,S)$. Let $\mathsf{Sph}(W)$ denote the union of all spherical parabolic subgroups $W_J$, with $J\subseteq S$. Is it true that there exists $x\in \mathsf{Sph}(W)$ such that $\varphi(x)=\boundW$?
\end{question}

If the answer to Question~\ref{q:spherical} is affirmative, then the possible values of $\boundW$ are severely restricted (see the following section).
%
%\begin{remark}
%There is an analogous conjecture for Kazhdan-Lusztig cells, asserting that every two sided cell intersects a finite parabolic subgroup (see \cite[\S13.12]{Lus:03}).
%\end{remark}
%

\subsection{Spherical Coxeter groups}

All weight functions on a spherical Coxeter group are bounded, and the statements in Corollary~\ref{cor1:groups} are trivial (for example, every finite set is regular). However in the spherical case one can ask more precise questions. In this section we explicitly determine the bound $\boundW$ and the cell $\cellW$ of a (necessarily bounded) weight function on a spherical Coxeter system. 

It is clear that if $\varphi(s)\leq 0$ for all $s\in S$ then $\boundW=0$ and $\cellW=W_{\varphi}^0$ (see Proposition~\ref{prop:basic1}). Thus we can assume that $\varphi(s)>0$ for some $s\in S$. The following lemma deals with the case $\varphi(s)\geq 0$ for all $s\in S$ (in the spherical case).

\begin{lemma}\label{lem:positivesphericalcase}
If $(W,S)$ is spherical and $\varphi(s)\geq 0$ for all $s\in S$ then $\boundW=\varphi(\sw_0)$ and $\cellW=\sw_0 W_{\varphi}^0$.
\end{lemma}

\begin{proof}
For all $x\in W$ we have $\sw_0=x(x^{-1}\sw_0)$ with $\ell(\sw_0)=\ell(x)+\ell(x^{-1}\sw_0)$, and hence $\varphi(x)=\varphi(\sw_0)-\varphi(x^{-1}\sw_0)\leq \varphi(\sw_0)$ (because $\varphi(x^{-1}\sw_0)\geq 0$). Thus $\boundW=\varphi(\sw_0)$. If $x\in \cellW$ then $\varphi(x)=\varphi(\sw_0)$ and so $\varphi(x^{-1}\sw_0)=0$. Thus $x^{-1}\sw_0\in W_{\varphi}^0$, and hence $x\in \sw_0 W_{\varphi}^0$. Conversely, if $x\in \sw_0 W_{\varphi}^0$ then $x=\sw_0y$ with $y\in W_{\varphi}^0$ and $\ell(x)=\ell(\sw_0)-\ell(y)$, and so $xy^{-1}=\sw_0$ with $\ell(x)+\ell(y^{-1})=\ell(\sw_0)$. Thus $\varphi(x)=\varphi(\sw_0)$, and so $x\in\cellW$. 
\end{proof}

Thus it remains to consider the case $\varphi(s)<0$ and $\varphi(t)>0$ for some $s,t\in S$. By Proposition~\ref{prop:wf}, in the cases $\sA_n$, $\sD_n$, $\sE_6$, $\sE_7$, $\sE_8$, $\sH_3$, $\sH_4$ and $\sI_2(2m+1)$, every weight function is constant on the generators, and so it remains to consider types $\sB_n$, $\sF_4$, and $\sI_2(2m)$. 

The case of dihedral groups $\sI_2(2m)$ is elementary, and we omit the simple proof of the following theorem.

\begin{thm}
Let $W=\langle s,t\mid s^2=t^2=(st)^{2m}=1\rangle$ be a dihedral group of order $4m$ and let $\varphi$ be a weight function. Let $\varphi(s)=a$ and $\varphi(t)=b$, and assume that $a<0$ and $b>0$. Then 
\begin{align*}
\boundW&=\begin{cases}
b&\text{if $a+b\leq 0$}\\
(m-1)a+mb&\text{if $a+b\geq 0$}
\end{cases}\\
\cellW&=\begin{cases}
\{t\}&\text{if $a+b<0$}\\
\{t(st)^k\mid 0\leq k\leq m-1\}&\text{if $a+b=0$}\\
\{\sw_0s\}&\text{if $a+b>0$}.
\end{cases}
\end{align*}
\end{thm}

We now turn attention to the $\sB_n$ case. We label the generators $s_1,\ldots,s_n$ with $(s_{n-1}s_n)^4=1$ (Bourbaki~\cite{Bou:02} conventions). 
Let $\varphi(s_1)=\cdots=\varphi(s_{n-1})=a$ and $\varphi(s_n)=b$.

\begin{thm}\label{thm:Bn} Let $\varphi$ be a weight function on $\mathsf{B}_n$ and let $\varphi(s_1)=a$ and $\varphi(s_n)=b$. Then
$$
\boundW=\max\bigg\{\frac{i(i-1)}{2}a+ib,\,\left(n(n-1)-\frac{(n-i)(n-i-1)}{2}\right)a+ib\,\bigg| \,0\leq i\leq n\bigg\},
$$
and if $a,b\neq 0$ then $\cellW= \{x\in X\mid \varphi(x)=\boundW\}$ where $X=\{x_i,y_i\mid 0\leq i\leq n\}$, where
$$
x_i=\prod_{j=1}^{i}(s_ns_{n-1}\cdots s_{n-i+j})\quad\text{and}\quad y_i=\sw_{S\backslash\{s_n\}}\prod_{j=1}^{i}(s_ns_{n-1}\cdots s_{j}).
$$
\end{thm}

\begin{proof} Let $J=S\backslash\{s_n\}$. From Proposition~\ref{prop:basic1}, if $b>0$ and $a<0$ (respectively $b<0$ and $a>0$) then the bound occurs on a minimal (respectively maximal) length $(W_J,W_J)$-double coset representative. These minimal (respectively maximal) length $(W_J,W_J)$-double coset representatives are $x_i$ (respectively $y_i$) for $0\leq i\leq n$. To see this, it is clear that $x_i$ (respectively $y_i$) is of minimal (respectively maximal) length in its $(W_J,W_J)$-double coset, and to check that we have found all minimal/maximal length representatives one can compute the cardinalities of the double cosets and verify that they cover~$W$; we omit the details. The result follows (note that $x_0=e$ and $y_n=\sw_0$, dealing with the cases $a,b\leq 0$ and $a,b\geq 0$).
\end{proof}

We now turn to the group $\sF_4$, with labelling
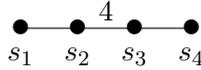
\begin{center}
\begin{tikzpicture}[scale=0.75,baseline=-0.5ex]
%\node at (0,0.5) {};
\node [inner sep=0.8pt,outer sep=0.8pt] at (3,0) (3) {\Large{$\bullet$}};
\node [inner sep=0.8pt,outer sep=0.8pt] at (4,0) (4) {\Large{$\bullet$}};
\node [inner sep=0.8pt,outer sep=0.8pt] at (5,0) (5) {\Large{$\bullet$}};
\node [inner sep=0.8pt,outer sep=0.8pt] at (6,0) (6) {\Large{$\bullet$}};
\node at (3,-0.5) {$s_1$};
\node at (4,-0.5) {$s_2$};
\node at (5,-0.5) {$s_3$};
\node at (6,-0.5) {$s_4$};
\node at (4.5,0.3) {$4$};
\draw (3,0)--(4,0);
%\draw (4,0.1)--(5,0.1);
%\draw (4,-0.1)--(5,-0.1);
\draw (4,0)--(5,0);
\draw (5,0)--(6,0);
%\draw (4.5-0.15,0.3) -- (4.5+0.08,0) -- (4.5-0.15,-0.3);%arrow
\end{tikzpicture}
\end{center}
By Proposition~\ref{prop:wf} every weight function $\varphi$ has $\varphi(s_1)=\varphi(s_2)$ and $\varphi(s_3)=\varphi(s_4)$. 

\begin{lemma}\label{lem:cosetsF4}
The elements of $W$ that are both maximal length $(W_{\{s_1,s_2\}},W_{\{s_1,s_2\}})$-double coset representatives, and minimal length $(W_{\{s_3,s_4\}},W_{\{s_3,s_4\}})$-double coset representatives are $121$, $121321$, $12132132$, $1213214321$, $121321432132$, $121323432132$, $121321324321$, $12132132432132$, $1213214321324321$, $121321324321324321$, and $121321324321323432132$. 
\end{lemma}

\begin{proof}
This is easily verified with the help of a computer~\cite{MAGMA}. 
\end{proof}

Let $A$ be the set consisting of the $11$ elements in Lemma~\ref{lem:cosetsF4}. Let $\sigma$ be the nontrivial diagram automorphism of $\sF_4$, and let $X=A\cup A^{\sigma}\cup\{e,\sw_0\}$ (thus $|X|=24$). 

\begin{prop}
Let $\varphi$ be a weight function on $\sF_4$ with $\varphi(s_1)=\varphi(s_2)=a$ and $\varphi(s_3)=\varphi(s_4)=b$. Then 
$$
\boundW=\max\{0,3a,3b,5a+b,a+5b,11a+7b,7a+11b,12a+9b,9a+12b,12a+12b\},
$$
and if $a,b\neq 0$ then 
$
\cellW=\{x\in X\mid \varphi(x)=\boundW\}
$ with $X$ as above.
\end{prop}

\begin{proof}
Suppose that $\varphi(s_1)=\varphi(s_2)=a>0$ and $\varphi(s_3)=\varphi(s_4)=b<0$. By Proposition~\ref{prop:basic1} if $x\in\cellW$ then $x$ is both a maximal length $(W_{\{s_1,s_2\}},W_{\{s_1,s_2\}})$-double coset representative, and a minimal length $(W_{\{s_3,s_4\}},W_{\{s_3,s_4\}})$-double coset representative. Thus the bound occurs on the set $A$ (and only on the set $A$). The values of $\varphi(x)$ with $x\in A$ are $3a$, $5a+b$, $6a+2b$, $7a+3b$, $8a+4b$, $7a+5b$, $8a+4b$, $9a+5b$, $10a+6b$, $11a+7b$, $12a+9b$. Since $a>0$ and $b<0$ the maximum is attained on one of the elements $3a$, $5a+b$, $11a+7b$, $12a+9b$. Thus
\begin{align*}
\boundW&=\max\{\varphi(x)\mid x\in A\}=\max\{3a,5a+b,11a+7b,12a+9b\},
\end{align*}
and $\cellW=\{x\in A\mid \varphi(x)=\boundW\}$. If $a<0$ and $b>0$ there is a dual argument (interchanging the roles of $a$ and $b$). Combining this with the cases $a<0$ and $b<0$ (where the bound is $0$, attained at $e$ only), and $a>0$ and $b>0$ (where the bound is $\varphi(\sw_0)$, attained only at $\sw_0$), the result follows. 
\end{proof}

\subsection{Affine Coxeter groups}

In this section we explicitly describe the cone of bounded weight functions on an irreducible affine Coxeter system~$(W,S)$. We do not use Theorem~\ref{thm1:cones} (or Corollary~\ref{cor1:groups}) directly -- instead we make use of the affine structure, thus avoiding the need to explicitly compute an automaton recognising $\cL(W,S)$. 

It is convenient to index the irreducible affine Coxeter systems as $\tilde{\sA}_n$ ($n\geq 2$), $\tilde{\sB}_n$ ($n\geq 3$), $\tilde{\sC}_n$ ($n\geq 1$), $\tilde{\sD}_n$ ($n\geq 4$), $\tilde{\sE}_n$ ($n=6,7,8$), $\tilde{\sF}_4$, and $\tilde{\sG}_2$ (in particular, the dimension $1$ affine group is denoted $\tilde{\mathsf{C}}_1$ rather than $\tilde{\mathsf{A}}_1$). We associate a root system $\Phi$ to $(W,S)$ as follows. If $W$ is of type $\tilde{\sX}_n$ with $\sX\neq\sC$ then let $\Phi$ be an irreducible root system of type $\sX_n$, while if $\sX=\sC$ let $\Phi$ be the (non-reduced) root system of type $\mathsf{BC}_n$. Let $\{\alpha_1,\ldots,\alpha_n\}$ be a fixed set of simple roots of $\Phi$, and let $\Phi^+$ be the associated positive roots. Let $Q$ be the coroot lattice, and $P$ the coweight lattice, associated to $\Phi$. Let $\omega_1,\ldots,\omega_n\in P$ be the fundamental coweights (defined by $\langle\omega_i,\alpha_j\rangle=\delta_{i,j}$), and let $P^+=\NN\omega_1+\cdots+\NN\omega_n$.

 There is a standard realisation of $W$ as a semidirect product $W=Q\rtimes W_0$ where $W_0$ is the associated spherical Weyl group (see \cite[\S1.1, \S1.2]{GLP:23}, and in particular \cite[Remark~1.1]{GLP:23} for conventions on the $\mathsf{BC}_n$ root system). Let $V$ be the underlying vector space of $\Phi$, and let $C_0$ be the fundamental (closed) alcove. 

It is convenient to work with the extended affine Weyl group $\Wext=P\rtimes W_0$. Write $t_{\lambda}\in\Wext$ for the translation by $\lambda\in P$ (thus $t_{\lambda}(v)=v+\lambda$ for $v\in V$). We have $\Wext=W\rtimes \Omega$, where $\Omega= P/Q$. We extend weight functions $\varphi:W\to\RR$ to functions $\varphi:\Wext\to \RR$ be setting $\varphi(g)=0$ for all $g\in\Omega$ (we call such an extension a \textit{weight function on $\Wext$}, however note that it may not be a weight function in the strict sense of Section~\ref{sec:definitions}).

Let $\varphi:\Wext\to\RR$ be a weight function. We define $\varphi(\alpha)$, for $\alpha\in\Phi$, as follows. If $\Phi$ is reduced we set $\varphi(\alpha_i)=\varphi(s_i)$ for all $1\leq i\leq n$, and if $\Phi$ is of type $\mathsf{BC}_n$ we define $\varphi(\alpha_i)=\varphi(s_i)$ for $1\leq i\leq n-1$, $\varphi(2\alpha_n)=\varphi(s_0)$, and $\varphi(\alpha_n)=\varphi(s_n)-\varphi(s_0)$. In all cases extend the definition to $\Phi$ by declaring $\varphi(\alpha)=\varphi(\beta)$ whenever $\beta\in W_0\alpha$, and it is convenient to set $\varphi(\beta)=0$ if $\beta\notin\Phi$. 

 Let
$$
\rho(\varphi)=\frac{1}{2}\sum_{\alpha\in\Phi^+}\varphi(\alpha)\alpha.
$$
For $x\in W_0$ let $\Phi(x)=\{\alpha\in\Phi^+\mid x^{-1}\alpha\in-\Phi^+\}$. 
%Let $\Phi_0=\{\alpha\in\Phi\mid \alpha/2\notin\Phi\}$ (so $\Phi_0=\Phi$ if $\Phi$ is reduced, while if $\Phi$ is of type $\mathsf{BC}_n$ then $\Phi_0$ is of type $\sB_n$). For $x\in W_0$ let $\Phi_0(x)=\{\alpha\in\Phi_0^+\mid x^{-1}\alpha\in-\Phi_0^+\}$. 

\begin{lemma}\label{lem:affine}
For $\la\in P^+$ and $x\in W_0$ we have
$$
\varphi(t_{\lambda})=\langle \la,2\rho(\varphi)\rangle\quad\text{and}\quad \varphi(x)=\sum_{\alpha\in\Phi(x)}\varphi(\alpha).
$$
\end{lemma}

\begin{proof}
The formula for $\varphi(x)$ with $x\in W_0$ follows by considering hyperplanes separating the alcove $C_0$ from the alcove $xC_0$ (in the non-reduced case, note that if $\alpha\in\Phi(x)$ and $2\alpha\in\Phi$ then $2\alpha\in\Phi(x)$ and $\varphi(\alpha)+\varphi(2\alpha)=\varphi(s_n)$). Similarly, the formula for $\varphi(t_{\la})$ follows by considering the hyperplanes separating $C_0$ from $t_{\la}(C_0)$. Since $\la\in P^+$, these hyperplanes are $H_{\alpha,k}$ with $1\leq k\leq \langle\lambda,\alpha\rangle$ for $\alpha\in\Phi^+$. In the reduced case, there is no double counting, and all hyperplanes in a parallelism class have the same weight $\varphi(\alpha)$. Thus, in the reduced case, 
$$
\varphi(t_{\lambda})=\sum_{\alpha\in\Phi^+}\langle\lambda,\alpha\rangle\varphi(\alpha)=\langle \lambda,2\rho(\varphi)\rangle.
$$
In the non-reduced case there is some double counting of the hyperplanes corresponding to the long and short roots (but not the middle length roots), but the definition of $\varphi$ ensures that after cancellation each hyperplane is counted with the appropriate weight (see \cite[Appendix~A]{Par:06b} for similar calculations in a related context). 
\end{proof}

Let 
$
V_0=\{v\in V\mid0\leq \langle v,\alpha_i\rangle\leq 1\text{ for all $i=1,\ldots,n$}\}
$
be \textit{Lusztig's box}, and let
$$
B_0=\{w\in\Wext\mid wC_0\subseteq V_0\}.
$$
The following theorem gives an explicit version of Corollary~\ref{cor1:groups} for extended affine Weyl groups.  

\begin{thm}\label{thm:affine1}
Let $\Wext$ be an extended affine Weyl group, and let $\varphi:\Wext\to\RR$ be a weight function. Then $\varphi$ is bounded if and only if $\langle\lambda,\rho(\varphi)\rangle\leq 0$ for all $\lambda\in P^+$. 

If $\varphi$ is bounded, then setting $\bb(\varphi)'=\max_{x\in W_0}\varphi(x)$ and $\bb(\varphi)''=\max_{y\in B_0}\varphi(y)$ we have $\boundW=\bb(\varphi)'+\bb(\varphi)''$, and
$$
\cellW=\{xt_{\la}y\mid x\in W_0,\,y\in B_0,\,\la\in P^+\text{ with }\varphi(x)=\bb(\varphi)',\,\varphi(y)=\bb(\varphi)'',\,\langle\la,\rho(\varphi)\rangle=0\}.
$$
Moreover, the bound $\boundW$ is attained on an element of the finite set $\{xy\mid x\in W_0,\,y\in B_0\}$. 
\end{thm}

\begin{proof}
If $\varphi$ is bounded, then for $\la\in P^+$ and $k\geq 0$ Lemma~\ref{lem:affine} gives $\varphi(t_{k\la})=k\langle\la,2\rho(\varphi)\rangle$, and hence $\langle\la,\rho(\varphi)\rangle\leq 0$ for all $\la\in P^+$. For the converse, it is well known and easy to prove that each $w\in \Wext$ can be written, in a unique way, as $w=xt_{\la}y$ with $x\in W_0$, $\la\in P^+$, and $y\in B_0$. Moreover, for all $x\in W_0$, $\la\in P^+$, and $y\in B_0$ we have $\ell(xt_{\la}y)=\ell(x)+\ell(t_{\la})+\ell(y)$. Thus
\begin{align}\label{eq:weights}
\varphi(w)=\varphi(x)+\langle\la,2\rho(\varphi)\rangle+\varphi(y),
\end{align}
and since $W_0$ and $B_0$ are finite sets it follows that if $\langle\la,2\rho(\varphi)\rangle\leq 0$ for all $\la\in P^+$ then $\varphi$ is bounded. The remaining statements now easily follow from (\ref{eq:weights}). 
\end{proof}

The irreducible affine Coxeter groups admitting non-constant weight functions are the $\tilde{\sB}_n$ ($n\geq 3$), $\tilde{\sC}_n$ ($n\geq 1$), $\tilde{\sF}_4$, and $\tilde{\sG}_2$ cases. Since $\tilde{\sG}_2$ is a special case of the triangle group $\Delta(2,3,2m)$ (with $m=3$, see Example~\ref{ex:232m}) we will not consider it further here. We consider the remaining cases below.

%
%\noindent\textbf{(1) The group $\tilde{\sG}_2$.} \james{This is special case of triangle group, so omit} Fix the labelling
%\begin{center}
%\begin{tikzpicture}[scale=0.75,baseline=-0.5ex]
%%\node at (0,0.5) {};
%\node [inner sep=0.8pt,outer sep=0.8pt] at (3,0) (3) {\Large{$\bullet$}};
%\node [inner sep=0.8pt,outer sep=0.8pt] at (4,0) (4) {\Large{$\bullet$}};
%\node [inner sep=0.8pt,outer sep=0.8pt] at (5,0) (5) {\Large{$\bullet$}};
%\node at (3,-0.5) {$0$};
%\node at (4,-0.5) {$2$};
%\node at (5,-0.5) {$1$};
%\node at (3,0.5) {$a$};
%\node at (4,0.5) {$a$};
%\node at (5,0.5) {$b$};
%\draw (3,0)--(4,0);
%\draw (4,0.1)--(5,0.1);
%\draw (4,-0.1)--(5,-0.1);
%\draw (4,0)--(5,0);
%\draw (4.5-0.15,0.3) -- (4.5+0.08,0) -- (4.5-0.15,-0.3);%arrow
%\end{tikzpicture}
%\end{center}
%
%The short roots are $1 0$, $11$, $21$, and the long roots are $01$, $31$, $32$. Thus 
%$$
%2\rho(\varphi)=(6a+4b)\alpha_1+(4a+2b)\alpha_2,
%$$
%and so $\varphi$ is bounded if and only if $3a+2b\leq 0$ and $2a+b\leq 0$. 
%
%The box consists of the 12 elements:
%$$
%e,0,02,021,0212,02120,02121,021210,0212102,02121021,021210212,0212102120,
%$$
%and so $\{\varphi(w)\mid w\in B_0\}$ equals
%$$
%\{0,a,2a,2a+b,3a+b,4a+b,3a+2b,4a+2b,5a+2b,5a+3b,6a+3b,7a+3b\}.
%$$
%Suppose that $a>0$. Then the conditions $3a+2b\leq 0$ and $2a+b\leq 0$ (the conditions for boundedness) imply that $2a$ is the maximum. Suppose that $a<0$. Then we get maximum~$0$. So, if $\varphi$ is bounded with $a,b\neq 0$ we have 
%$$
%\max_{v\in B_0}\varphi(v)=\begin{cases}2a&\text{if $a>0$}\\
%0&\text{if $a<0$}.
%\end{cases}
%$$
% 
% 

\begin{example}[The group $\tilde{\mathsf{F}}_4$]
 
Consider the group $\tilde{\sF}_4$. Using Bourbaki conventions the roots $\alpha_1$ and $\alpha_2$ are long in $\Phi$, and the affine generator $s_0$ satisfies $(s_0s_1)^3=1$. Let $\varphi(s_0)=\varphi(s_1)=\varphi(s_2)=a$ and $\varphi(s_3)=\varphi(s_4)=b$. Direct calculation gives
$$
2\rho(\varphi)=(10a+6b)\alpha_1+(18a+12b)\alpha_2+(24a+18b)\alpha_3+(12a+10b)\alpha_4
$$
and so by Theorem~\ref{thm:affine1} the weight function $\varphi$ is bounded if and only if 
$
5a+3b\leq 0$, $3a+2b\leq 0$, $4a+3b\leq 0$, and $6a+5b\leq 0$ (these inequalities arise by considering $\langle\omega_i,2\rho(\varphi)\rangle\leq 0$ for $i=1,2,3,4$). The inequalities $3a+2b\leq 0$ and $4a+3b\leq 0$ are redundant, and so we see that a weight function on $\tilde{\mathsf{F}}_4$ is bounded if and only if $5a+3b\leq 0$ and $6a+5b\leq 0$. 
%To see this, suppose that $5a+3b\leq 0$ and $6a+5b\leq 0$. Note that $3a+2b=\frac{3}{5}(5a+3b)+\frac{1}{5}b$ and $4a+3b=\frac{4}{5}(5a+3b)+\frac{3}{5}b$ and so if $b\leq 0$ we have $3a+2b\leq 0$ and $4a+3b\leq 0$. Similarly, $3a+2b=\frac{1}{2}(6a+5b)-\frac{1}{2}b$ and $4a+3b=\frac{2}{3}(6a+5b)-\frac{1}{3}b$ and so if $b\geq 0$ then $3a+2b\leq 0$ and $4a+3b\leq 0$.
\end{example}
%
%\begin{prop} A weight function $\varphi$ on $\tilde{\sF}_4$ is bounded if and only if $5a+3b\leq 0$ and $6a+5b\leq 0$. 
%\end{prop}
%
%\begin{proof}
%The short positive roots (in the root basis) are $
%0 0 1 0$, $0 0 0 1$, $0 1 1 0$, $0 0 1 1$,  $1 1 1 0$, $0 1 1 1$, $1 1 1 1$, $0 1 2 1$, $1 1 2 1$, $1 2 2 1$, $1 2 3 1$, $1 2 3 2$, and the long roots are $1 0 0 0$, $0 1 0 0$, $1 1 0 0$, $0 1 2 0$, $1 1 2 0$, $1 2 2 0$, $0 1 2 2$, $1 1 2 2$, $1 2 2 2$, $1 2 4 2$, $1 3 4 2$, $2 3 4 2$. Hence
%$$
%2\rho(\varphi)=(10a+6b)\alpha_1+(18a+12b)\alpha_2+(24a+18b)\alpha_3+(12a+10b)\alpha_4
%$$
%and so by Theorem~\ref{thm:affine} the weight function $\varphi$ is bounded if and only if 
%$
%5a+3b\leq 0$, $3a+2b\leq 0$, $4a+3b\leq 0$, and $6a+5b\leq 0$. The inequalities $3a+2b\leq 0$ and $4a+3b\leq 0$ are redundant. To see this, suppose that $5a+3b\leq 0$ and $6a+5b\leq 0$. Note that $3a+2b=\frac{3}{5}(5a+3b)+\frac{1}{5}b$ and $4a+3b=\frac{4}{5}(5a+3b)+\frac{3}{5}b$ and so if $b\leq 0$ we have $3a+2b\leq 0$ and $4a+3b\leq 0$. Similarly, $3a+2b=\frac{1}{2}(6a+5b)-\frac{1}{2}b$ and $4a+3b=\frac{2}{3}(6a+5b)-\frac{1}{3}b$ and so if $b\geq 0$ then $3a+2b\leq 0$ and $4a+3b\leq 0$.
%\end{proof}
%

\begin{example}[The group $\tilde{\mathsf{B}}_n$]
Consider the group $\tilde{\sB}_n$, $n\geq 3$. If $\varphi$ is a weight function then $\varphi(s_0)=\varphi(s_1)=\cdots=\varphi(s_{n-1})=a$ and $\varphi(s_n)=b$. Using standard Bourbaki conventions, the short roots are $e_k$, $k=1,\ldots,n$, and the long roots are $e_i-e_j,e_i+e_j$ with $1\leq i<j\leq n$. Thus we have
\begin{align*}
2\rho(\varphi)&=\sum_{i=1}^n[2(n-i)a+b]e_i.
\end{align*}
Since $\omega_j=e_1+\cdots+e_j$ Theorem~\ref{thm:affine1} gives that $\varphi$ is bounded if and only if 
$
(2n-j-1)a+b\leq 0
$
%$$
%\sum_{i=1}^j[2(n-i)a+b]\leq 0
%$$
for all $j=1,\ldots,n$. It follows that a weight function is bounded if and only if $2(n-1)a+b\leq 0$ and $(n-1)a+b\leq 0$. 
\end{example}

\begin{example}[The group $\tilde{\mathsf{C}}_n$]
Consider the group $\tilde{\sC}_n$. Let $n\geq 2$. If $\varphi$ is a weight function we write $\varphi(s_1)=\cdots=\varphi(s_{n-1})=b$, $\varphi(s_n)=a$, and $\varphi(s_0)=c$. %Up to the obvious symmetry, we may assume that $a\geq c$.
In the standard setup of the $\mathsf{BC}_n$ root system (see~\cite[Remark~1.1]{GLP:23}) we compute
$$
\rho(\varphi)=\frac{1}{2}\sum_{i=1}^n(a+c+2(n-i)b)e_i.
$$
Since $\omega_i=e_1+\cdots+e_i$ for $1\leq i\leq n$ Theorem~\ref{thm:affine1} gives that $\varphi$ is bounded if and only if
$a+c+(2n-i-1)b\leq 0$ for all $1\leq i\leq n$. It follows that a weight function $\varphi$ on $\tilde{\sC}_n$ is bounded if and only if $a+c+2(n-1)b\leq 0$ and $a+c+(n-1)b\leq 0$. 
\end{example}

\subsection{Bounded representations of Hecke algebras}\label{subsec:5}

Let $(W,S)$ be a Coxeter system, and let $\psi:W\to \ZZ_{\geq 0}$ be a non-negative integer valued weight function on $(W,S)$ with $\psi(s)> 0$ for all $s\in S$. Recall the definition of the associated weighted Hecke algebra $H=H(W,S,\psi)$ from the introduction. 
We now give the proof of Corollary~\ref{cor1:reps}.

\begin{proof}[Proof of Corollary~\ref{cor1:reps}]
Let $\pi$ be a $1$-dimensional representation. For $s\in S$ the quadratic relation $T_s^2=1+(\sq^{\psi(s)}-\sq^{-\psi(s)})T_s$ implies that $\pi(T_s)\in\{\sq^{\psi(s)},-\sq^{-\psi(s)}\}$, and if $m_{s,t}$ is odd then the relation $T_sT_tT_s\cdots =T_tT_sT_t\cdots$ ($m_{s,t}$ terms on each side) forces $\pi(T_s)=\pi(T_t)$. Conversely, each choice $\pi(T_s)\in\{\sq^{\psi(s)},-\sq^{-\psi(s)}\}$ with $\pi(T_s)=\pi(T_t)$ whenever $m_{s,t}$ is odd extends uniquely to a $1$-dimensional representation by Matsumoto's Theorem~\cite{Mat:64}. It follows that the map $\varphi:W\to\ZZ$ with $\varphi(x)=\deg\pi(T_x)$ is a $\ZZ$-valued weight function on $W$, and that $\varphi$ is bounded if and only if the representation $\pi$ is bounded. Moreover, $\Gamma(\pi)=\Gamma_{W,S}(\varphi)$, and the result now follows from Corollary~\ref{cor1:groups}. 
\end{proof}

\begin{example}
Let $W=\Delta(2,4,6)$ be the hyperbolic triangle group generated by $S=\{s,t,u\}$ (see Example~\ref{ex:246}). Let $\psi:W\to \ZZ_{\geq 0}$ be the weight function with $\psi(s)=\psi(t)=\psi(u)=1$, and let $H=H(W,S,\psi)$ (the equal parameter Hecke algebra). By the proof of Corollary~\ref{cor1:reps}, there is a $1$-dimensional representation $\pi$ of $H$ with $\pi(T_s)=-\sq^{-1}$, $\pi(T_t)=\sq$, and $\pi(T_u)=-\sq^{-1}$. The weight function $\varphi(x)=\deg\pi(T_x)$ for $x\in W$ is the weight function $\varphi=\varphi_3$ from Example~\ref{ex:246}, and by the proof of Corollary~\ref{cor1:reps} the bound of $\pi$ is $\bb(\pi)=\bb_{W,S}(\varphi_3)=1$ and the cell recognised by $\pi$ is $\Gamma(\pi)=\Gamma_{W,S}(\varphi_3)$ (the regular set recognised by the automaton in Figure~\ref{fig:246a}).
\end{example}

%
%The following conjecture is attributed to Bill Casselman (see \cite{Cas:94,Cas:94b} and \cite{BGS:14,Gun:10}). 
%
%\begin{conjecture}[Casselman]
%If $\Ga$ is a left, right, or two-sided cell in the weighted Coxeter system $(W,S,\psi)$ then $\{w\in\cL(W,S)\mid \sfp(w)\in\Ga\}$ is a regular language. 
%\end{conjecture}
%

\bibliographystyle{plain}
%\bibliography{/Users/jamesp/Dropbox/Research/bibtex/Parkinson}

\bigskip
  \footnotesize

  \noindent J\'er\'emie Guilhot, \textsc{Institut Denis Poisson, Universit\'e de Tours,
    37200 Tours, France}\par\nopagebreak
 \noindent \textit{E-mail address:} \texttt{Jeremie.Guilhot@lmpt.univ-tours.fr}
 \smallskip
 
  \noindent Eloise Little, \textsc{School of Mathematics and Statistics, University of Sydney, NSW 2006, Australia}\par\nopagebreak
 \noindent \textit{E-mail address:} \texttt{E.Little@maths.usyd.edu.au}
\smallskip

 \noindent James Parkinson, \textsc{School of Mathematics and Statistics, University of Sydney, NSW 2006, Australia}\par\nopagebreak
 \noindent \textit{E-mail address:} \texttt{jamesp@maths.usyd.edu.au}

\end{document}